\documentclass[a4paper,11pt,reqno]{amsart}
\usepackage[utf8]{inputenc}
\usepackage{mathtools}
\usepackage{amssymb}
\usepackage{amsfonts}
\usepackage{amsthm}
\usepackage{bbm}
\usepackage{fullpage}
\usepackage{color}
\usepackage{hyperref}
\usepackage{stmaryrd}
\usepackage{enumitem}
\usepackage{accents}

\newcommand{\N}{\mathbb N}
\newcommand{\Z}{\mathbb Z}
\newcommand{\R}{\mathbb R}

\newcommand{\eps}{\varepsilon}
\renewcommand{\P}{\mathbb P}
\newcommand{\E}{\mathbb E}
\newcommand{\1}{\mathbbm 1}
\newcommand{\0}{\mathbf 0}
\newcommand{\F}{\mathcal F}

\newtheorem{theorem}{Theorem}[section]
\newtheorem{lemma}[theorem]{Lemma}
\newtheorem{proposition}[theorem]{Proposition}
\newtheorem{corollary}[theorem]{Corollary}
\newtheorem{thmx}{Theorem}

\theoremstyle{remark}

\newtheorem{remark}[theorem]{Remark}

\theoremstyle{definition}

\newcommand{\isDistr}{\overset d=}
\newcommand{\dd}{\text{d}}

\newcommand{\cev}[1]{\accentset{\leftarrow}{#1}}
\newcommand{\A}{\mathcal A}

\newcommand{\q}{{q^*}}
\newcommand{\p}{{p^*}}

\newcommand{\ceil}[1]{{\left\lceil {#1}\right\rceil}}

\renewcommand{\a}{\mathfrak{a}}

\newcommand{\GFF}{\operatorname{GFF}}

\newcommand{\PSRW}{P^{\operatorname{SRW}}}
\newcommand{\ESRW}{E^{\operatorname{SRW}}}

\newcounter{constants}
\newcommand{\newconstant}
{
\refstepcounter{constants}\ensuremath{c_{\theconstants}}
}
\newcommand{\oldconstant}[1]{\ensuremath{c_{\ref{#1}}}}

\newcommand{\X}{\mathcal X}

\newcommand{\Visit}{V}
\newcommand{\VVisit}{\widehat V}

\newcommand{\eeta}{\eta_1}
\newcommand{\eeeta}{\eta_2}

\newcommand{\localization}{exceptional }
\newcommand{\localizationn}{``exceptional'' }
\title{Fluctuations of partition functions of directed polymers in weak disorder beyond the $L^2$-phase}
\author{Stefan Junk}
\address{AIMR, Tohoku University. 2-1-1 Katahira, Aoba-ku, Sendai, 980-8577 Japan}
\email{sjunk@tohoku.ac.jp}
\date{\today}

\begin{document}

\lineskip=0pt

\begin{abstract}
We study the directed polymer model in a bounded environment in weak disorder  without $L^2$-boundedness, specifically the speed of homogenization for the field $(W_n^{x})_{x\in\Z^d}$, where $W_n^{x}$ denotes the associated martingale for the polymer starting from $x$. We show that a suitably re-centered spatial average over a set of diameter $n^{1/2}$ convergence to zero at rate $n^{-\xi+o(1)}$, where the exponent is an explicit function of the inverse temperature $\beta$.
\end{abstract}

\maketitle

\section{Introduction}

\subsection{Motivation}\label{sec:motivation}

The directed polymer model describes random paths, called \emph{polymers} in this context, in a medium with random impurities. Much information about the long-term behavior of the polymer is encoded in an associated martingale $(W_n)_{n\in\N}$. Specifically, it is known that the movement is diffusive if and only if the almost sure limit $W_\infty\coloneqq \lim_{n\to\infty}W_n$ is positive, which occurs in spatial dimension $d\geq 3$ at high enough temperatures. This situation is referred to as \emph{weak disorder} and it is the focus of the present article.

\smallskip The \emph{$L^2$-weak disorder phase} refers to a subset of the weak disorder phase characterized by $L^2$-boundedness of $(W_n)_{n\in\N}$. In contrast to the implicit condition $W_\infty>0$ used to define weak disorder, $L^2$-boundedness is straightforward to check since the second moment can be expressed in terms of the moment generating function of the overlap $\sum_n\1_{X_n=Y_n}$ between two independent random walks $(X_n)_{n\in\N}$ and $(Y_n)_{n\in\N}$. This observation greatly simplifies calculations and, as a result, there are many more papers that analyze this regime than the more general weak disorder phase. It is, however, known that $L^2$-boundedness only holds in a subset of the weak disorder phase.

\smallskip A number of results first proved under this condition have later been extended to the whole weak disorder phase, most notably the central limit theorem for the polymer endpoint, but our understanding of the $L^2$-weak disorder phase is much more complete. For example, it is known that $W_\infty$ is an analytic function of the inverse temperature \cite[Theorem 6.2]{CY06} and the rate of convergence of $W_n$ to $W_\infty$ can be computed explicitly \cite[Theorem 1.1]{CN21}. It is natural to wonder whether such results can be extended to the whole weak disorder phase with different techniques. At present, it is not clear whether the gap in knowledge between the two regimes is due to the limitations of our current methods or whether a quantifiable change in behavior occurs between weak disorder and $L^2$-weak disorder. 

\smallskip  In this paper, we give some evidence for the latter hypothesis. Namely, our main result concerns the speed of homogenization of the field $(W_n^{0,x})_{n\in\N,x\in\Z^d}$, where $x$ indicates the starting point of the polymer. After taking a spatial average over a set of diameter $\sqrt n$ and re-centering, we show that the speed of convergence is $n^{-\xi+o(1)}$, where the exponent $\xi$ is different from the corresponding exponent in the $L^2$-weak disorder phase.

\subsection{Definition of the model}\label{sec:formal}

A recent survey of the model can be found in \cite{C17}. Let $((\omega_{t,x})_{(t,x)\in\N\times\Z},\F,\P)$ be an i.i.d. family of real-valued weights, called \emph{environment}, satisfying  
\begin{align}\label{eq:exp_mom}
\E\big[e^{\beta|\omega_{0,0}|}\big]<\infty\quad\text{ for all }\beta\geq 0.
\end{align}
We write $\F_k\coloneqq\sigma(\omega_{t,x}:t\leq k)$ for the natural filtration of $\omega$. The energy of a path $\pi$ in time-interval $I\subset\R_+$ is defined by
\begin{align}\label{eq:def_H}
H_I(\omega,\pi)\coloneqq \textstyle \sum_{i\in I\cap\N} \omega_{i,\pi_i},
\end{align}
with $H_n\coloneqq H_{[1,n]}$. For a parameter $\beta\geq 0$, called the \emph{inverse temperature}, the \emph{polymer measure} $\mu_{\omega,n}^\beta$ is defined by
\begin{align*}
\mu_{\omega,n}^\beta(\dd X)\coloneqq (Z_n^\beta)^{-1}e^{\beta H_n(\omega,X)}\PSRW(\dd X),
\end{align*}
where $(X,\PSRW)$ denotes the simple random walk and $Z_n^\beta\coloneqq \ESRW[e^{\beta H_n(\omega,X)}]$ denotes the normalizing constant, called the \emph{partition function}. That is, $\mu_{\omega,n}^\beta$ is a perturbation of $\PSRW$ such that paths are attracted by sites with positive weight and repelled by negative ones. Next, we introduce the associated martingale $(W_n^\beta)_{n\in\N}$ mentioned in Section~\ref{sec:motivation},
\begin{align*}
W_n^\beta\coloneqq \frac{Z_n^\beta}{\E[Z_n^\beta]}=\ESRW[e^{\beta H_n(\omega,X)-n\lambda(\beta)}],\qquad\text{ where }\lambda(\beta)\coloneqq \log\E[e^{\beta\omega_{0,0}}].
\end{align*}
In view of \eqref{eq:exp_mom}, it is not hard to see that the almost sure limit $W_\infty^\beta\coloneqq \lim_{n\to\infty}W_n^\beta$ satisfies a zero-one law, $\P(W_\infty^\beta>0)\in\{0,1\}$, and we distinguish between weak disorder \eqref{eq:WD} and strong disorder \eqref{eq:SD} accordingly,
\begin{align}
\P(W_\infty^\beta>0)=1,\tag{WD}\label{eq:WD}\\ 
\P(W_\infty^\beta=0)=1.\tag{SD}\label{eq:SD}
\end{align}

In the following theorem, we collect some known results about the transition from \eqref{eq:SD} to \eqref{eq:WD} as well as some basic information about the behavior within those regimes.
\begin{thmx}\label{thmx:phase}
\begin{itemize}
 \item[(i)]In dimensions $d=1$ and $d=2$, \eqref{eq:SD} holds for all $\beta>0$. In dimensions $d\geq 3$, there exists $\beta_{cr}=\beta_{cr}(d)\in(0,\infty)$ such that \eqref{eq:WD} holds for $\beta\in[0,\beta_{cr})$ and \eqref{eq:SD} holds for $\beta\in(\beta_{cr},\infty)$. 
 \item[(ii)]In dimensions $d\geq 3$, there exists $\beta_{cr}^{L^2}\in(0,\beta_{cr})$ such that $(W_n^\beta)_{n\in\N}$ is $L^2$-bounded if and only if $\beta\in[0,\beta_{cr}^{L^2})$. 
 \item[(iii)] Weak disorder \eqref{eq:WD} implies that $(W_n^\beta)_{n\in\N}$ is uniformly integrable and that the polymer measure satisfies a central limit theorem in probability, i.e., for every $f\colon\R^d\to\R$ bounded and continuous,
 \begin{align*}
\sum_{x\in\Z^d}f(x/\sqrt n)\mu_{\omega,n}^\beta(X_n=x)\xrightarrow[n\to\infty]P\int_{\R^d}f(x)\varphi(x)\dd x,
\end{align*}
where $\varphi$ is the standard normal density.
 \item[(iv)] Strong disorder \eqref{eq:SD} implies that the polymer measure localizes, i.e., there exists $c>0$ such that $\limsup_{n\to\infty}I_n\geq c$ almost surely, where 
\begin{align}\label{eq:I_n}
I_n\coloneqq \textstyle\sum_{x\in\Z^d}\mu_{\omega,n-1}(X_n=x)^2.
\end{align}
\end{itemize}
\end{thmx}

Parts (i) and (iii) are proved  in \cite{CY06}. The $L^2$-phase $[0,\beta_{cr}^{L^2})$ was introduced in \cite{B89,IS88}. The inequality $\beta_{cr}^{L^2}\leq \beta_{cr}$ is clear, whereas the proof for the strict inequality was given in a number of papers, see \cite[Remark 5.2]{C17} for the precise references. Finally, part (iv) was first proved in a Brownian environment in \cite{CH06} and later extended to \emph{stochastic linear evolution} in \cite{Y10}, which generalizes the current setting.

\smallskip An important characteristic of the model will be the critical exponent,
\begin{align}\label{eq:def_p}
\p(\beta)\coloneqq\sup\left\{p\geq 1:(W_n^\beta)_{n\in\N}\text{ is }L^p\text{ bounded}\right\}.
\end{align}
Clearly $\p(\beta)=1$ in strong disorder and $\p(\beta)\geq 2$ in $L^2$-weak disorder, $\beta<\beta_{cr}^{L^2}$. Beyond that, we note that even though $(W_n^\beta)_{n\in\N}$ is uniform integrability in weak disorder by Theorem~\ref{thmx:phase}(iii), it may still be the case that $\p(\beta)=1$. The following extra assumption guarantees that this does not occur: we say that the environment is upper bounded if 
\begin{align}\label{eq:bounded}\tag{U-bd.}
\P\big(\omega_{t,x}\in(-\infty, K]\big)=1\qquad\text{ for some }K>0.
\end{align}
The necessity of this assumption for our result is discussed in the beginning of Section~\ref{sec:limitations}. 
\begin{thmx}\label{thmx:moments}
	Assume \eqref{eq:bounded}. Then \eqref{eq:WD} implies $\p(\beta)>1$. Moreover for $\beta>\beta_{cr}^{L^2}$ and any $t>1$, it holds that
\begin{align}\label{eq:tail}
\P\left({\textstyle \sup_n W_n^\beta>t}\right)\geq \oldconstant{c:tail}t^{-\p},
\end{align}
where $\newconstant\label{c:tail}\coloneqq e^{-2\beta K}/2$.
\end{thmx}

Theorem~\ref{thmx:moments} follows from \cite[Theorem~2.1]{J21_1}. The proof can be found in the appendix.

\subsection{Main result and related literature}\label{sec:result}

For $x\in\Z^d$, we define the partition function started at time $t$ from $x$ in the time interval $(s,t]$ by
\begin{align}\label{eq:def_Wtx}
	W_t^{s,x}\coloneqq W_{t-s}^\beta\circ\theta_{s,x},
\end{align}
where $\theta_{t,x}$ denotes the space-time shift acting on the environment. For a compactly supported function $f\colon\R^d\to\R$, we consider 
\begin{align*}
\X^f_n\coloneqq n^{-d/2}\sum_{x\in\Z^d}f(x/\sqrt n)\big(W_n^{0,x}-1\big).
\end{align*}
We recall some results about homogenization of $(W_n^{0,x})_{x\in\Z^d}$ in weak disorder. 
\begin{thmx}\label{thmx:ew}
Assume $d\geq 3$ and \eqref{eq:exp_mom}.
\begin{itemize}
 \item [(i)]Assume \eqref{eq:WD}. For every continuous, compactly  supported $f\colon\R^d\to\R$,
 \begin{align}\label{eq:decay}
\X_n^f\xrightarrow[n\to\infty]{L^1}0.
\end{align}
\item[(ii)] Assume $\beta<\beta_{cr}^{L^2}$. For every continuous, compactly  supported $f\colon\R^d\to\R$,
 \begin{align}\label{eq:CNN}
n^{\frac{d-2}4}\X_n^f\xrightarrow[n\to\infty]d\int_{\R^d}f(x)\GFF^{\gamma(\beta)}(x)\dd x,
\end{align}
where $(\GFF^\gamma(x))_{x\in\R^d}$ is the Gaussian free field of intensity $\gamma$ and
\begin{align*}
\gamma(\beta)\coloneqq \left(e^{\lambda(2\beta)-2\lambda(\beta)}-1\right)\E\left[(W_\infty^\beta)^2\right].
\end{align*}
\end{itemize}
\end{thmx}

Part (i) has been proved in a  continuous setting in \cite[Theorem~2.1]{CNN20} and for completeness we give a short proof for the discrete setting in the appendix. Part (ii) is proved in \cite[Theorem~2.5]{CNN20} in the continuous setting and in \cite[Theorem~1.1]{LZ20} in our discrete setting. We also mention the earlier works \cite{MSZ16,GRZ18,CCM19} on \eqref{eq:CNN} in the continuous setting under stronger assumptions on $\beta$.

\smallskip By Theorem~\ref{thmx:ew}(ii), the field $(W_n^{0,x})_{x\in\Z^d}$ homogenizes upon taking a spatial average on the diffusive scale and the rate of convergence is $n^{-(d-2)/4}$. The exponent is independent of $\beta$ but the intensity $\gamma(\beta)$ of the limiting Gaussian free field diverges as $\beta$ approaches $\beta_{cr}^{L^2}$, which suggests that the rate of convergence is slower in the remainder of the weak disorder phase. Our main result confirms this.

\begin{theorem}\label{thm:main}
Assume \eqref{eq:bounded}, \eqref{eq:WD} and $\beta>\beta_{cr}^{L^2}$. Let
\begin{align*}
	\xi(\beta)\coloneqq \frac d2-\frac{2+d}{2\p(\beta)}
\end{align*}
and let $f\colon\R^d\to\R$ be continuous and compactly supported with $f\not\equiv 0$. For every $\eps>0$,
\begin{align}\label{eq:main}
\lim_{n\to\infty}\P\left(n^{-\xi(\beta)-\eps}\leq |\X_n^{f}|\leq n^{-\xi(\beta)+\eps}\right)=1.
\end{align}
\end{theorem}

\begin{remark}
 From Proposition~\ref{prop:a}(iv), we see that $\p(\beta)<2$ under the assumptions of Theorem~\ref{thm:main} and therefore $\xi(\beta)<\frac{d-2}4$. On the other hand, it may be the case that $\xi(\beta)=0$.

\end{remark}

We now give a heuristic explanation for the change in the exponent outside of the $L^2$-phase. As will become clear in the next section, the correlation between $W_n^{0,x}$ and $W_n^{0,y}$ is proportional to the probability that the polymers from $x$ and $y$ meet before time $n$. If we presume that a so-called local limit theorem is valid, then the probability of meeting in a space-time point $(t,z)$ should be comparable to $n^{-d}(\widehat{W}^{t,z}_\infty)^2$, where $n^{-d}$ comes from the hitting probability under the simple random walk and $(\widehat{W}^{t,z}_\infty)_{t,z\in [0,n]\times[-n^{1/2},n^{1/2}]^d}$ is a family of independent copies of $W_\infty$. The latter contribution encodes the effect of the environment around the common endpoint $(t,z)$. Thus, the correlation can be approximated as  
\begin{align*}
	|W^{0,x}W^{0,y}-1|\approx n^{-d}\textstyle{\sum_{t,z\in [0,n]\times[-n^{1/2},n^{1/2}]^d}}(\widehat{W}^{t,z}_\infty)^2,
\end{align*}
see also \eqref{eq:corrector_approx} below. The difference between the $L^2$-regime and the remainder of the weak disorder phase is whether this sum satisfies a law of large numbers or not. In the first case, the correlation is comparable to the number of summands (which recovers the exponent from Theorem~\ref{thmx:ew}(ii)), otherwise the sum is dominated by a few large terms, which corresponds to the existence of space-time sites  whose hitting probability is $\gg n^{-d/2}$.  Note that the local limit theorem, which we used to justify this approximation, is only known for the $L^2$-phase, see \cite{S95,V06} for the precise statement.

\smallskip We also record the following consequence of Theorem~\ref{thm:main}.

\begin{corollary}\label{cor:p}
Assume \eqref{eq:bounded} and \eqref{eq:WD}. Then $\p(\beta)\geq 1+\frac{2}d$. In particular, $\beta\mapsto\p(\beta)$ is discontinuous at $\beta_{cr}$.
\end{corollary}

\begin{proof}
By Theorem~\ref{thmx:ew}(i) and Theorem~\ref{thm:main}, we must have $\xi(\beta)\geq 0$ in weak disorder and thus $\p(\beta)\geq 1+\frac 2d$. On the other hand, it is clear that $\p(\beta)=1$ in strong disorder.
\end{proof}

Corollary~\ref{cor:p} seems surprising at first sight, but in fact the same value turns out to be critical for a related polymer model, called \emph{directed polymers in $\gamma$-stable random environment}. We now briefly explain this connection. 

\smallskip The model can be defined similarly to our setup, but instead of \eqref{eq:exp_mom} one assumes that $\E[e^{\beta\omega_{t,x}}]=\infty$ for some finite $\beta>0$. It is then convenient to re-parametrize the model and replace the exponential weight $e^{\beta H_n(\omega,\pi)}$ of a path $\pi$ by $\prod_{t=1}^n (1+\beta \eta_{t,\pi(t)})$, where $\beta\in[0,1]$ and where the random environment $(\eta_{t,x})_{t,x}$ is centered, supported on $(-1,\infty)$ and satisfies $\P(\eta_{0,0}\geq u)\sim u^{-\gamma}$ as $u\to\infty$, for some $\gamma>0$. This representation actually goes back to the earliest works on directed polymers but has recently reappeared in \cite{V21}, where it was shown that the model exhibits a non-trivial phase transition in $\beta$ if and only if $\gamma>1+\frac 2d$, see \cite[Theorems 1.4--1.6]{V21}.

\smallskip It is interesting to note that the lower bound from Corollary~\ref{cor:p} matches the critical value from for heavy-tailed environment even though our environment is bounded. One can guess that the heavy-tailed model can be recovered from our model by a rescaling argument. That is, we consider the simple random walk evaluated at times $0,\ell_n,2\ell_n,\dots,n/\ell_n$ for some suitable time-scale $\ell_n$ and assign to a path $X_0,X_{\ell_n},\dots,X_{n/\ell_n}$ a weight $\approx\prod_{i=1}^{n/\ell_n} W^{i\ell_n,X_{i\ell_n}}_{(i+1)\ell_n}$, where $W^{s,x}_t$ was defined in \eqref{eq:def_Wtx}. In view of \eqref{eq:tail}, the field $(W_{(i+1)\ell_n}^{i\ell_n,x})_{i\in\N,x\in\Z^d}$ should behave, up to some short-range dependence, like an i.i.d. field of $\p$-stable random variables and, if $\ell_n$ grows sufficiently slowly, $(X_0,X_{\ell_n},X_{2\ell_n},\dots)$ should behave like a random walk with finite range, so we would recover the setup of \cite{V21}. 

\smallskip Furthermore, the comparison with the $\gamma$-stable random environment suggests that the lower bound from Corollary~\ref{cor:p} might be sharp.

\smallskip One interesting question for future research is whether such an approximation can be used to construct an intermediate disorder regime for the directed polymer in dimensions $d\geq 3$. In dimensions $d=1$ and $2$, it is by now well-understood that the polymer measure has a non-trivial scaling limit if we choose a time-dependent inverse temperature $\beta=\beta_n$ that decays to zero at the appropriate rate, but to the best of our knowledge it is not known whether a similar phase appears in dimension $d\geq 3$ as $\beta_n\downarrow \beta_{cr}$. In the heavy-tailed setup, the existence of an intermediate disorder phase has recently been proved in \cite{BL21} and successive works, where they study the case of strong disorder, $\gamma<1+\frac 2d$, and inverse temperature 
\begin{align*}
\beta_n=n^{-\frac{d}{2\gamma}(1+\frac d2-\gamma)+o(1)}.
\end{align*}

\subsection{Strategy}\label{sec:strategy}

Here and in the rest of the paper, we simplify the notation by replacing index sets over space, time or space-time with continuous sets, with the understanding that an intersection with $\Z^d$, $\N$ or $\N\times\Z^d$ has to be taken, for example in \eqref{eq:iid} below. 

\smallskip The idea for the lower bound is that $\X^f_n$ behaves approximately like a sum of $N_n$ independent random variables, each of which has a decent probability of taking a value larger than $n^{-\xi(\beta)-o(1)}$ in absolute value. The sequence $(N_n)_{n\in\N}$ satisfies $\lim_{n\to\infty}N_n=\infty$, so the claim follows from existing results about the anti-concentration of independent random variables. More precisely, we show that on an event with large probability, $\X_n^f$ can be decomposed as
\begin{align}\label{eq:decomp}
\X_n^f=A+B+{\textstyle \sum_{i=1}^{N_n}}Z_i,
\end{align}
where $B$ is negligibly small, $A$ is measurable with respect to some sigma field $\mathcal G$ and $Z_0,\dots,Z_{N_n}$ are independent conditionally on $\mathcal G$. See Proposition~\ref{prop:decomposition} for the exact statement. To establish anti-concentration, we need to know that the $Z_i$ are sufficiently dispersed, i.e., there exists some deterministic $c>0$ such that, almost surely,
\begin{align}\label{eq:anticonc}
\sup_{\lambda\in\R} \P\Big(Z_i\in \big[\lambda,\lambda+n^{-\xi(\beta)-o(1)}\big]\Big|\mathcal G\Big)\leq 1-c.
\end{align}
To construct the decomposition \eqref{eq:decomp}, we first show the existence of so-called \emph{\localization sites}, i.e. space-time sites $(t,x)\in[0,n]\times[-n^{1/2},n^{1/2}]^d$ that have probability at least $n^{-\xi(\beta)-o(1)}$ of being visited by a polymer started from $\{0\}\times[-n^{1/2},n^{1/2}]^d$. Given such a \localization site $(t,x)$, we let $Z_i$ be the contribution to $\X_n^f$ from paths visiting $(t,x)$. In that way, modifying $\omega_{t,x}$ changes the value of $Z_i$ by an amount proportional to the weight of $(t,x)$, and hence \eqref{eq:anticonc} holds with $\mathcal G$ the sigma-field generated by the environment outside of $\omega_{t,x}$. At this point, we do not go into further details,  but we note that the actual construction ensures that we can choose the same $\mathcal G$ for all $Z_i$ simultaneously and that modifying the environment at the \localization site corresponding to $Z_i$ does not influence $(Z_j)_{j\neq i}$. For the purpose of this introduction, we ignore these technical difficulties and focus on the existence of the \localization sites. 

\smallskip The main idea is that the probability that a polymer starting from $\{0\}\times[-n^{1/2},n^{1/2}]^d$ visits $(t,x)$ depends mostly on the environment close to $(t,x)$, specifically on the value of the backward partition function $\cev W_{t-\ell_n}^{t,x}$, where $\ell_n$ is a small time-scale and, for $0\leq s\leq t$,
\begin{align}\label{eq:def_reverse}
\cev W_{s}^{t,x}\coloneqq {\cev E}{}^{t,x}\left[e^{\beta H_{[s,t)}(\omega,X)-(t-s)\lambda(\beta)}\right].
\end{align}
Here, $(X=(X_k)_{k=t,\dots,0},{\cev P}{}^{t,x})$ denotes the simple random walk running backward in time, starting from space-time site $(t,x)$. We will choose $\ell_n$ small enough that we can extract $n^{1+d/2-o(1)}$ independent copies of $W_{\ell_n}$ from the family 
\begin{align}\label{eq:iid}
\big(\cev W_{t-\ell_n}^{t,x}\big)_{(t,x)\in[0,n]\times[-n^{1/2},n^{1/2}]^d}.
\end{align}
The \localization sites correspond to near-maximizers in \eqref{eq:iid}, the order of which can be determined with the help of extreme value statistics. To do so, we need a lower tail bound on $W_{\ell_n}$. Such a result is usually proved by large deviation methods and we therefore  introduce the logarithmic moment generating function of $(\log W_n)_{n\in\N}$,
\begin{align}\label{eq:def_a}
\a^\beta(p)\coloneqq \lim_{n\to\infty}\frac 1n\log\E\big[(W_n^\beta)^p\big],
\end{align}
and the critical exponent for exponential growth of moments,
\begin{align}\label{eq:def_q} 
\q(\beta)\coloneqq \inf\big\{p\geq 1:\a^\beta(p)>0\big\}.
\end{align}
 The existence of the limit \eqref{eq:def_a} is proved in Proposition~\ref{prop:a}. The G\"artner-Ellis theorem shows that if  $\ell_n\geq c\log n$, then
\begin{align}\label{eq:ldp}
\P\Big(W_{\ell_n}\geq n^{\frac{1+d/2}\q}\Big)\gtrapprox  n^{-1-d/2+o(1)},
\end{align}
hence
\begin{align*}
	\max_{(t,x)\in[0,n]\times[-n^{1/2},n^{1/2}]^d}  \cev W_{t-\ell_n}^{t,x}\geq n^{(1+d/2)/\q}
\end{align*}
with high probability. We call $(t,x)$ a \localization site if $\cev W_{t-\ell_n}^{t,x}\geq n^{(1+d/2)/\q-o(1)}$. A number of technical estimates based on \eqref{eq:WD} and the central limit theorem, Theorem~\ref{thmx:phase}(iii), are necessary to see that this definition satisfies the property outlined above, i.e., that $\{X_t=x\}$ has large probability under the polymer measure starting from $\{0\}\times[-n^{1/2},n^{1/2}]^d$ in the time-horizon $[0,n]$.

\smallskip Regarding the upper bound, it is natural to use the $L^p$-bound from Theorem~\ref{thmx:moments}. We thus consider the martingale $(M_{n,k}^f)_{k=0,\dots,n}$ defined by 
\begin{align}\label{eq:def_MG}
M^f_{n,k}\coloneqq n^{-d/2}{\textstyle \sum_{x\in\Z^d}}f(x/\sqrt n)(W_k^{0,x}-1),
\end{align}
whose quadratic variation $k\mapsto\langle M_{n,\cdot}^f\rangle_k$ can be approximated by
\begin{align}\label{eq:corrector_approx}
\langle M^f_{n,\cdot}\rangle_n\approx n^{-d}\sum_{(t,x)\in[1,n]\times[-n^{1/2},n^{1/2}]^d} \big(\cev W_{1}^{t,x}\big)^2,
\end{align}
see Proposition~\ref{prop:corrector}. If we assume $\p>1$, then a straightforward argument based on decomposing the summands in  \eqref{eq:corrector_approx} according to their size yields $\langle M^f_{n,\cdot}\rangle_n\lessapprox n^{-d+(2+d)/\p}$, which we then combine with the general relation $|M_{n,n}|\approx \langle M_{n,\cdot}\rangle_n^{1/2}$. We refer to Section~\ref{sec:upper_corrector} for a detailed description and summarize the above outline in the following theorem. 

\begin{theorem}\label{thm:upper_lower}
Assume $d\geq 3$, \eqref{eq:WD} and let $f$ be as in Theorem~\ref{thm:main}.
\begin{enumerate}
\item[(i)]\textbf{Lower bound}: Recall \eqref{eq:def_q}. For every $\eps>0$,
\begin{align}
\lim_{n\to\infty}\P\left(|\X^f_n| \leq n^{-\frac d2+\frac{d+2}{2\q}-\eps}\right)&=0.\label{eq:lower_bound_mg}
\end{align}
\item[(ii)]\textbf{Upper bound}: Recall \eqref{eq:def_p}. If $\p(\beta)>1$, then, for every $\eps>0$,
\begin{align}
\lim_{n\to\infty}\P\left(\langle M^f_{n,\cdot}\rangle_n\geq n^{-d+\frac{2+d}{\p\wedge 2}+\eps}\right)=0,\label{eq:upper_bound}\\
\lim_{n\to\infty}\P\left(|\X_n^f|\geq n^{-\frac d2+\frac{2+d}{2\p\wedge 4}+\eps}\right)=0.\label{eq:upper_bound_mg}
\end{align}
\end{enumerate}
\end{theorem}
Note that it is possible to obtain a lower bound for the quadratic variation complementing \eqref{eq:upper_bound}, which was done in an earlier version of this paper \cite{J22_old}. We also emphasize that the assumptions \eqref{eq:bounded} and $\beta>\beta_{cr}^{L^2}$ from Theorem~\ref{thm:main} are not necessary up to this point. They are, however, necessary for the final part of the argument, which is to show that the two bounds agree, i.e., that $\p=\q$.

\begin{theorem}\label{thm:pq}
Assume $d\geq 3$, \eqref{eq:WD}, $\beta>\beta_{cr}^{L^2}$ and \eqref{eq:bounded}. For every $\eps>0$, there exist $\newconstant\label{c:5}>1$ and $\newconstant\label{c:useless}>0$ such that, for all $n$ large enough,
\begin{align}\label{eq:claim}
\E[W_n^{\p+\eps}]\geq \oldconstant{c:useless}\oldconstant{c:5}^n.
\end{align}
In particular, $\p=\q$.
\end{theorem}

Now the main result, Theorem~\ref{thm:main}, follows from Theorems~\ref{thm:upper_lower} and~\ref{thm:pq}, together with Theorem~\ref{thmx:moments}.

\smallskip The method used to prove Theorem~\ref{thm:pq} is quite interesting in its own right and gives new insight into the weak disorder phase without $L^2$-boundedness, but we postpone this discussion to the beginning of Section~\ref{sec:pq}. 

\subsection{Limitations and extensions}\label{sec:limitations}

We first discuss the necessity of the assumptions in Theorem~\ref{thm:main}. Concerning the behavior at $\beta_{cr}^{L^2}$, it is natural to conjecture that in this case the decay rate is the same as in the $L^2$-bounded case, $|\X^f_n|\approx n^{-\frac{d-2}4+o(1)}$, with a subpolynomial correction. In fact, this would follow from Theorem~\ref{thm:upper_lower} if we knew that $\q(\beta_{cr}^{L^2})= \p(\beta_{cr}^{L^2})=2$. On the one hand, the argument in \cite{BT10,BS10} for the inhomogeneous pinning model shows that, for all $\eps\in(0,\eps_0)$, 
\begin{align*}
\sup_n\E\Big[\big(W_n^{\beta_{cr}^{L^2}}\big)^{2-\eps}\Big]<\infty
\end{align*}
and hence $\p(\beta_{cr}^{L^2})=2$. However, we cannot exclude the possibility that $\q(\beta_{cr}^{L^2})>2$, so it is not clear that the bounds in Theorem~\ref{thm:upper_lower} agree. Our proof of $\p(\beta)=\q(\beta)$ cannot be extended to $\beta\leq\beta_{cr}^{L^2}$, see Remark~\ref{rm:critical}, although one would naturally expect that this equality is true in the whole weak disorder phase. 

\smallskip  The assumption \eqref{eq:bounded} is not necessary for the lower bound, Theorem~\ref{thm:upper_lower}(i), and for the upper bound we only need to assume $\p(\beta)>1$. Recently, the conclusion of Theorem~\ref{thmx:moments} has been extended to a large class of unbounded environments in \cite{FJ23} and Theorem~\ref{thm:upper_lower}(ii) thus continues to hold if the environment satisfies \cite[Condition 1]{FJ23}. On the other hand, it seems difficult to remove the assumption \eqref{eq:bounded} in our proof of Theorem~\ref{thm:pq}. %a way that seems difficult to 

\smallskip In another direction, Theorem~\ref{thm:main} only reveals the rate of convergence in \eqref{eq:decay} up to an error of order $n^{o(1)}$ and we hope that a more precise statement similar to Theorem~\ref{thmx:ew}(ii) can be proved in the future. The argument in this paper strongly suggests that the limiting object would be a suitable stable version of the Gaussian Free field, see also the discussion at the beginning of Section~\ref{sec:upper}. However, it seems to be quite difficult to even define such a ``stable free field'', see \cite[Open Problem 6.3]{BPR19}. 

\smallskip There are essentially two steps of the argument where new ideas seem necessary in order to obtain a more precise result:

\begin{itemize}
 \item[(1)] First, for the construction of \localization sites we want to treat \eqref{eq:iid} as an i.i.d. family, but this is only true if the index set is replaced by $\mathcal G_n\subseteq [0,n]\times[-n^{1/2},n^{1/2}]^d$, where $\mathcal G_n$ satisfies $|(t,x)-(s,y)|_\infty\gg\ell_n$ for distinct $(t,x),(s,y)\in\mathcal G_n$. This thinning introduces an error of poly-logarithmic size. To improve upon it, one would need to show that if $1\ll |(t,x)-(s-y)|_\infty\ll \ell_n$,  then the probability that $\cev W_{t-\ell_n}^{t,x}$ and $\cev W_{s-\ell_n}^{s,y}$ are both large is much smaller than the probability that only one of them is large.
 
 \item[(2)] In addition, we would need a better error control in the large deviation lower bound \eqref{eq:ldp}, which would in turn  require a better understanding of $\a$. In this work, we can learn enough about $\a$ from general principles, but to go beyond the $n^{o(1)}$ precision it would be helpful to understand where $\a$ is differentiable and whether the rate of convergence can be improved from $\E[W_n^p]=e^{n(\a(p)+o(1))}$ to $\E[W_n^p]=e^{n\a(p)}(1+o(1))$, for $p>\p$. Note that the last assertion is known for the special case $p=2$, since $\E[W_n^2]$ is equivalent to the partition function of the homogeneous pinning model, see \cite[Theorem~2.2$\text{(1)}$]{G07}.
\end{itemize}

Finally, we note that in $L^2$-weak disorder there is a result analog to Theorem~\ref{thmx:ew}(ii) for the log-partition functions, see \cite[Corollary 2.11]{CNN20} and \cite[Theorem~1.2]{LZ20}, which in particular shows
\begin{align}\label{eq:log_convergence}
n^{-d/2}\sum_{x\in[-n^{1/2},n^{1/2}]^d} \left(\log W_n^{0,x}-\E[\log W_n]\right)\asymp n^{-(d-2)/4}.
\end{align}
We refer to \cite{CNN20,LZ20} for the motivation due to the connection between $\log W_n$ and the KPZ equation.

\smallskip In the $L^2$-regime, it is known that $W_\infty^\beta$ has all negative moments, see \cite[Proposition 1]{M10}, and thus $\E[\log Z_n]$ is bounded as $n\to\infty$. This result has recently been extended to the full weak disorder phase in \cite[Theorem~1.1$\text{(iv)}$]{J21_1} and it is therefore natural to wonder what one can say about the left-hand side of \eqref{eq:log_convergence} without $L^2$-boundedness. Note that, a priori, there is no reason to expect the rate of convergence to be the same as in Theorem~\ref{thm:main}, since $(W_n^{0,x})_{x\in\Z^d}$ and $(\log W_n^{0,x})_{x\in\Z^d}$ are dominated, respectively, by the upper and the lower tail of $W_n^\beta$. To illustrate that our methods do not easily apply to this question, let us try to repeat the analysis for the upper bound using the quadratic variation in the case $f=\1_{[-1,1]^d}$. Namely, we write
\begin{align*}
n^{-d/2}{\textstyle \sum_{x\in[-n^{1/2},n^{1/2}]^d}}(\log W_n^{0,x}-\E[\log W_n])=\widetilde M_{n,n}-\widetilde M_{n,0},
\end{align*}
where $(\widetilde M_{n,m})_{m=0,\dots,n}$ is defined by $\widetilde M_{n,m}\coloneqq n^{-d/2}\sum_{x\in[-n^{1/2},n^{1/2}]^d}\E[\log W_n^{0,x}|\F_m]$. Its quadratic variation $\langle \widetilde M_{n,\cdot}\rangle_n$ equals 
\begin{align*}
n^{-d}\sum_{m=1}^n\E\left[\Bigg(\sum_{x\in[-n^{1/2},n^{1/2}]^d} \E\Big[\log \frac{W_n^{0,x}}{W_{m-1}^{0,x}}\Big|\F_{m}\Big]- \E\Big[\log \frac{W_n^{0,x}}{W_{m-1}^{0,x}}\Big|\F_{m-1}\Big]\Bigg)^2 \Bigg|\F_{m-1}\right],
\end{align*}
but this formula is much more complicated than \eqref{eq:corrector_approx} and we do not know how to analyze it. 

\subsection{Outline and conventions}

Section~\ref{sec:lower} contains the proof of the lower bound, Theorem~\ref{thm:upper_lower}(i). We obtain a lower tail bound for $W_n^\beta$ in Section~\ref{sec:tails} and define a sequence of \localization sites in Section~\ref{sec:local}, which we then use in Section~\ref{sec:proof_lower} to prove the lower bound. The proof of the upper bound, Theorem~\ref{thm:upper_lower}(ii), can be found in Section~\ref{sec:upper}. We first compute the quadratic variation (Section~\ref{sec:corrector}), then obtain an upper bound on it (Section~\ref{sec:upper_corrector}) and obtain the conclusion in Section~\ref{sec:proof_upper}. Finally, Section~\ref{sec:pq} contains the proof of Theorem~\ref{thm:pq}. In the appendix, we provide the reference for Theorem~\ref{thmx:moments} and give a short proof of Theorem~\ref{thmx:ew}(i).

\smallskip In addition to the convention about index sets mentioned at the start of Section~\ref{sec:strategy}, we follow the convention that constants $c_1,c_2,\dots$ are fixed throughout the article, while constants $c,c',c'',\dots$ are only used within a proof. All constants are positive. 

\smallskip We will occasionally refer to times $t\in\R$ or to intervals $I\subseteq\R_+$, with the understanding that an integer part should be takes, i.e., they should be interpreted as $\lceil t\rceil$ and $I\cap\N$. A similar convention is applied to sites $x\in\R^d$ and sets $A\subseteq\R^d$.

\section{Proof Theorem~\ref{thm:upper_lower}: Lower bound}\label{sec:lower}

\subsection{Tail bounds}\label{sec:tails}
In this section we prove the lower tail bound \eqref{eq:ldp} for $W_n^\beta$. First, we check some easy properties of the logarithmic moment generating function $\a$. The results in this section are valid in any dimension and regardless of whether \eqref{eq:WD} holds.

\begin{proposition}\label{prop:a}
Recall the definitions of $\a$, $\p$ and $\q$ in \eqref{eq:def_a}, \eqref{eq:def_p} and \eqref{eq:def_q}.
\begin{itemize}
	\item[(i)]The limit in \eqref{eq:def_a} is well-defined and $\a(p)\in[0,\lambda(p\beta)-p\lambda(\beta)]$.
 \item[(ii)]The function $p\mapsto \a(p)$ is non-decreasing, convex and continuous.
 \item[(iii)] It holds that $\p\leq\q$. 
 \item[(iv)] If $d\geq 3$ and $\beta>\beta_{cr}^{L^2}$, then $\q<2$.
\end{itemize}
\end{proposition}
\begin{proof}
The existence of the limit \eqref{eq:def_a} for $p\geq 1$ follows from subadditive Lemma. Indeed, by Jensen's inequality,
\begin{alignat*}{3}
\E\big[W_{n+m}^p\big]&=\E\left[W_m^p\E\Big[\Big(\sum_x\mu_{\omega,m}^\beta(X_m=x)W_n\circ\theta_{m,x}\Big)^p\Big|\F_m\Big]\right]&&\\
&\leq\E\left[W_m^p\E\Big[\sum_x\mu_{\omega,m}^\beta(X_m=x)\big(W_n\circ\theta_{m,x}\big)^p\Big|\F_m\Big]\right]&=&\E\big[W_{m}^p\big]\E\big[W_{n}^p\big].
\end{alignat*}
For $p\in[0,1]$ the inequality is reversed and we apply the superadditive Lemma instead. Moreover, for $p>1$ Jensen's inequality implies
\begin{align*}
	1=\E[W_n^\beta]^p\leq \E\big[(W_n^\beta)^p\big]\leq \E\Big[E\big[e^{\sum_{k=1}^np\beta\omega_{k,X_k}-pn\lambda(\beta)}\big]\Big]=e^{n(\lambda(p\beta)-p\lambda(\beta))},
\end{align*}
which completes the proof of (i). For (ii), we again apply Jensen's inequality we get $\a(q)\geq \tfrac qp\a(p)$ for $q\geq p$, hence $\a$ is non-decreasing in $\R_+$ and strictly increasing in $(\q,\infty)$. The convexity (and hence continuity) of $\a$ follows easily from H\"older's inequality,
\begin{align*}
\E\big[W_n^{\lambda p+(1-\lambda)q}\big]\leq \E\big[W_n^p\big]^{\lambda }\E\big[W_n^q\big]^{1-\lambda}\quad\text{ for }\lambda\in[0,1],
\end{align*}
and $\p\leq \q$ is clear from the definition. Finally, assume that $d\geq 3$ and $\beta>\beta_{cr}^{L^2}$, which is equivalent to $\lambda(2\beta)-2\lambda(\beta)>-\log(p_{\operatorname{return}}(d))$, where $p_{\operatorname{return}}(d)$ is the probability that two independent simple random walks in $\Z^d$ meet after time $1$ (see \cite[Theorem 3.3]{C17} and references therein). We can choose $T\in\N$ and $\eps>0$ such that 
\begin{align*}
\lambda(2\beta)-2\lambda(\beta)\geq -\log(p_{\operatorname{return},T}(d))+\log(1+\eps),
\end{align*}
where $p_{\operatorname{return},T}(d)\coloneqq P^{\operatorname{SRW},\otimes 2}(X_n^1=X_n^2\text{ for some }n\in\{1,\dots,T\})$. Thus, by Fubini's theorem,
\begin{align*}
\E\big[(W_{kT}^\beta)^2\big]&=E^{\operatorname{SRW},\otimes 2}\left[e^{(\lambda(2\beta)-2\lambda(\beta))\sum_{n=1}^{kT}\1_{X^1_n=X^2_n}}\right]\\
&\geq \left(e^{\lambda(2\beta)-2\lambda(\beta)}p_{\operatorname{return},T}(d)\right)^k\\
&\geq (1+\eps)^k
\end{align*}
and therefore $\a(2)\geq \frac 1T\log (1+\eps)>0$. By continuity, we get $\a(p)>0$ for some $p<2$.
\end{proof}

Next, we obtain a lower tail bound from the G\"artner-Ellis theorem.

\begin{proposition}\label{prop:ldp_lower}
For every $\delta>0$ there exist $x>0$ and $t_0$ such that, for all $t\geq t_0$,
\begin{align}\label{eq:lower}
\P\left(W_t^\beta\geq e^{tx}\right)\geq e^{-tx\q(1+\delta)}.
\end{align}
\end{proposition}

\begin{proof}
Fix $\eta>0$ to be chosen later. Let $\a^*(x)\coloneqq \sup_{q\geq 0}\{qx-\a(q)\}$ denote the Legendre transform of $\a$ and recall from \cite[Definition 2.3.3]{DZ} the definition of an exposed point of $\a^*$. Since the function $\a$ is convex, it is differentiable except for at most countably many points. Thus, we find $q\in[\q(1+\eta),\q(1+2\eta)]$ such that $\a$ is differentiable at $q$. Moreover, $q>\q$ implies that $\a(q)>0$ and $\a'(q)>0$. By \cite[Lemma 2.3.9]{DZ}, $\a'(q)$ is an exposed point with exposing hyperplane $q$ and 
\begin{align*}
\a^*(\a'(q))=q\a'(q)-\a(q)\leq q\a'(q)\leq \q\a'(q)(1+2\eta).
\end{align*}
Let $x\coloneqq\a'(q)(1-\eta)$ and note that, by \cite[Theorem~2.3.6]{DZ}, 
\begin{align*}
\liminf_{t\to\infty}\frac 1t\log\P\left(\log W_t>tx\right)&\geq -\inf_{y\in\F:y>\a'(q)(1-\eta)}\a^*(y)\\
&\geq -\q\a'(q)(1+2\eta)\\
&=-\q x\frac{1+2\eta}{1-\eta},
\end{align*}
where $\F$ denotes the set of exposed points of $\a^*$. We now choose $\eta>0$ small enough that $\frac{1+2\eta}{1-\eta}<1+\delta$, so that \eqref{eq:lower} holds for $t$ large enough.
\end{proof}

Before stating the next lemma, we introduce notation for the restricted partition function,
\begin{align}\label{eq:def_restricted}
	W_n[\1_A]=E[e^{\beta H_n(\omega,X)-n\lambda(\beta)}\1_A(X)],
\end{align}
where $A$ is measurable with respect to the sigma field of the simple random walk $(X_m)_{m\in\N}$. Similar notation will be used for the shifted partition function $W_t^{s,x}$ and the backward partition function $\cev W^{t,x}_s$. 

\smallskip The following lemma is similar to the argument used in \cite[Theorem~1.1$\text{(i)}$]{J21_1}. 
\begin{lemma}\label{lem:concave}
For every $m,n\in\N$, $a>0$ and $A\subseteq\Z^d$, almost surely,
\begin{align*}
\P\left(\frac{W_{m+n}[\1_{X_{m+n}\in A}]}{W_m}>a\Big|\F_m\right)\geq \min_{|x|_1\leq m}2\P(W_n[\1_{X_n\in A-x}]>2a)-1
\end{align*}
\end{lemma}

\begin{proof}
Consider the concave function $g_a(x)\coloneqq (x/a-1)\wedge 1$ and note that, on $\R_+$,
\begin{align}\label{eq:approx}
\1_{(a,\infty)}\geq g_a\geq 
2\1_{(2a,\infty)}-1.
\end{align}
Hence
\begin{align*}
&\P\left(\frac{W_{m+n}[\1_{X_{m+n}\in A}]}{W_m}>a\Big|\F_m\right)\\
&=\E\left[\1_{(a,\infty)}\Big(\textstyle\sum_{|x|_1\leq m}\mu_{\omega,m}^\beta(X_m=x) W_n[\1_{X_n\in A-x}]\circ\theta_{m,x}\Big)\Big|\F_m\right]\\
&\geq \E\left[g_a\Big(\textstyle\sum_{|x|_1\leq m}\mu_{\omega,m}^\beta(X_m=x) W_n[\1_{X_n\in A-x}]\circ\theta_{m,x}\Big)\Big|\F_m\right]\\
&\geq \E\left[\textstyle\sum_{|x|_1\leq m}\mu_{\omega,m}^\beta(X_m=x)g_a\left(W_n[\1_{X_n\in A-x}]\circ\theta_{m,x}\right)\Big|\F_m\right]\\
&\geq \min_{|x|_1\leq m}\E\left[g_a\left(W_n[\1_{X_n\in A-x}]\right)\right]\\
&\geq \min_{|x|_1\leq m}2\P\left(W_n[\1_{X_n\in A-x}]>2a\right)-1.
\end{align*}
We used \eqref{eq:approx}  in the first and last inequality while the second inequality is Jensen's inequality.
\end{proof}

\subsection{Construction of \localization sites}\label{sec:local}

In this section, we prove two technical estimates, Lemmas~\ref{lem:heavy_existence} and \ref{lem:heavy_prop}. Intuitively, they guarantee the existence of $n^{o(1)}$ \localization sites, as introduced in Section~\ref{sec:strategy}, i.e., space-time areas that are likely to be visited by polymers starting from $\{0\}\times[-n^{1/2},n^{1/2}]^d$. The definition of such an area is local, i.e. it depends on a space-time area of diameter $n^{o(1)}$, which means that we have independence between areas that are far apart and we can thus apply extreme value statistics, as explained in Section~\ref{sec:strategy}.

\smallskip A drawback of the local construction is that we need to ensure that our notion of \localizationn is not destroyed by the remainder of the environment. This is intuitively believable, since we only need the environment everywhere else to behave in a typical manner, but the formal proof is rather technical. 

\smallskip We start the construction by defining some constants that depend on $f$. The purpose of these is to define a set of time-space sites $\mathtt T_n\times\mathtt S_n$ such that the expectation of $f(X_0/\sqrt n)$ under the backward polymer measure started from $(t,x)$ (see \eqref{eq:bw_pm}) can be controlled uniformly in $(t,x)\in\mathtt T_n\times\mathtt S_n$, see the proof of Lemma~\ref{lem:heavy_prop} for details. First, there exists $z\in\R^d$ such that 
\begin{align}\label{eq:def_z}
	f*\varphi(z)\coloneqq \int_{\R^d}f(x)\varphi(z-x)\dd x\neq 0,
\end{align}
where $\varphi$ denotes the density of the standard normal distribution. Indeed, the Fourier transform $\hat\varphi$ of $\varphi$ is a again Gaussian and in particular non-zero everywhere, while $f$ is by assumption non-trivial, hence $\hat f\not\equiv 0$. It is well-known that the Fourier transform of $f * \varphi$ is the product of $\hat f$ and $\hat \varphi$. In particular, $\hat f\hat\varphi$ and consequently $f*\varphi$ are non-trivial, which yields \eqref{eq:def_z}. Next, we note that since $f$ is uniformly continuous, we find $\eeta\in(0,1)$ such that, for all $x,y\in\R^d$, 
\begin{align}\label{eq:def_eta}
|x-y|\leq2 \eeta\quad \implies\quad |f(x)-f(y)|\leq |f*\varphi(z)|/8.
\end{align}
Finally, we choose $\eeeta>0$ such that
\begin{align}\label{eq:def_eta2}
	\max\Big\{\big|1-\sqrt{ 1-\eeeta}\big|,\Big|1-\frac{1}{\sqrt{1-\eeeta}}\Big|\Big\}<\frac{\eeta}{\max\{L,\|z\|_\infty\}},
\end{align}
where $L$ is large enough that the support of $f$ is contained in the interior of $[-L,L]^d$. 

\smallskip Proceeding with the construction, let $\ell_n\coloneqq \lfloor \log^2(n)\rfloor$ and consider the grid $\mathtt T_n\times\mathtt S_n$, where 
\begin{align}\label{eq:def_ST}
	\mathtt T_n&\coloneqq  (\ell_n+1)\Z,\\
	\mathtt S_n&\coloneqq   \big((2\ell_n+1)\Z^d\big)\cap \big(z  n^{1/2} +\big[-\eeta n^{1/2},\eeta n^{1/2}\big]^d\big).
\end{align}
The elements of $\mathtt T_n\times\mathtt S_n$ constitute the starting locations and the purpose of ``$\ell_n$'' in \eqref{eq:def_ST} is to ensure that the field $(\cev W^{t,x}_{t-\ell_n})_{(t,x)\in\mathtt T_n\times \mathtt S_n}$ is independent (see Figure~\ref{fig:construction}), so that we can study the extreme values. More precisely, for $k=1,\dots,\eeeta n^{2\delta}/2$, we write
\begin{align*}
\mathtt T_n(k)\coloneqq \mathtt T_n\cap\left((1-\eeeta)n+\big[(2k-1)n^{1-2\delta},2kn^{1-2\delta}\big]\right).
\end{align*}
and introduce the time $T_k$ where $(W^{t,x}_{t-\ell_n})_{(t,x)\in\mathtt T_n(k)\times\mathtt S_n}$ achieves a near-maximum,
\begin{align*}
T_k\coloneqq \max\Big\{t\in \mathtt T_n(k):\exists y\in\mathtt S_n\text{ s.t. }\cev W_{t-\ell_n}^{t,y}\geq n^{\frac{2+d}{2\q}-\eps/2}\Big\}.
\end{align*}
The maximum of the empty set is defined to be $-\infty$. On $T_k>-\infty$, we define $Y_k$ as the location of the near-maximizer. This definition may however not be unique and, for technical reasons, it is convenient to choose $Y_k$ uniformly at random among the candidates, see Figure~\ref{fig:construction}. More precisely, we slightly enlarge the probability space to include an i.i.d. sequence $(U_k)_{k\in\N}$ whose marginals are uniformly distributed on $[0,1]$ and which is independent of everything else. Then, if there are $N$ sites $y\in\mathtt S_n$ such that $\cev W_{T_k-\ell_n}^{T_k,y}\geq n^{\frac{2+d}{2\q}-\eps/2}$, we define $Y_k$ to be the $\lceil NU_k\rceil$-largest such site in the lexicographical order on $\mathtt S_n$.  Let also
\begin{figure}[h]
	\includegraphics[width=.9\textwidth]{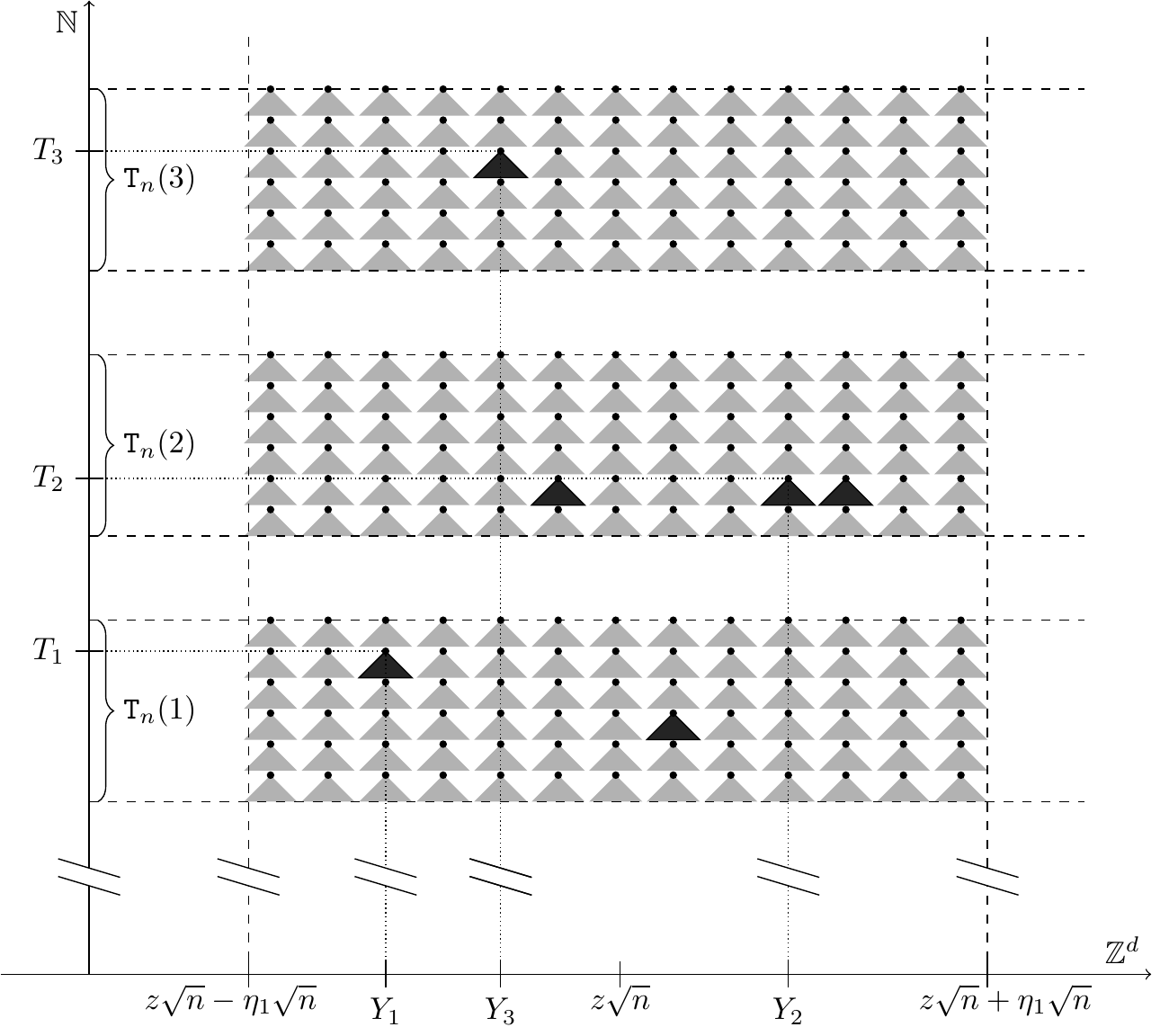}
	\caption{Illustration for the construction of \localization sites in $d=1$. The black dots represent the grid $\mathtt T_n\times \mathtt S_n$ and the shaded areas indicate the area accessible for the partition functions $(W^{t,x}_{t-\ell_n})_{t,x\in\mathtt T_n\times \mathtt S_n}$. By definition, the cones are disjoint and the partition functions are thus independent. The dark cones are the sites $(t,x)$ with $W^{t,x}_{t-\ell_n}\geq n^{(2+d)/2\q-\eps/2}$. The middle row illustrates the case where there more than one such maximizer and $Y_2$ is thus chosen at random among the three candidates. Note that no such draw occurs in the bottom row, since the maximizers have different time-coordinates and in this case the latest one is chosen. \label{fig:construction}}
\end{figure}
\begin{align*}
E_n\coloneqq \Big\{T_k>-\infty\text{ for all }k=1,\dots,\eeeta n^{2\delta}/2\Big\}.
\end{align*}
Note that on $E_n$ we have $n^{o(1)}$ sites $(T_i,Y_i)$ with $\cev W^{T_i,Y_i}_{T_i-\ell_n}=n^{\frac{2+d}{2\q}-o(1)}$. Except for the sub-polynomial error-term, this is the same order one can expect for an independent family of random variables with tail exponent $\q$. Some technical consequences of the construction are summarized in the next lemma:
\begin{lemma}\label{lem:properties}
The above construction satisfies the following properties:
\begin{itemize}
 \item[(i)] For $t\in\mathtt T_n$, $\{T_k\geq t\}$  is measurable with respect to
	 \begin{align*}
		 \F_{[t-\ell_n,\infty)\setminus \mathtt T_n}\coloneqq \sigma\big(\omega_{s,x}:s\in [t-\ell_n,\infty)\setminus \mathtt T_n\big).
	 \end{align*}
 \item[(ii)] Conditional on $\sigma(T_1,\dots,T_{\eeeta n^{2\delta}/2},Y_1,\dots,Y_{\eeeta n^{2\delta}/2})$, on $E_n$, the sequence
	 \begin{align}\label{eq:seq}
	 \big(\omega_{T_1,Y_1},\dots,\omega_{T_{\eeeta n^{2\delta}/2},Y_{\eeeta n^{2\delta}/2}}\big)
 \end{align}
 has the same law as the unconditioned environment, i.e., it is i.i.d. with law $\P(\omega_{0,0}\in\cdot)$.
\item[(iii)] Conditional on $E_n$, $(Y_1,\dots,Y_{\eeeta n^{2\delta}/2})$ is independent and uniformly distributed on $\mathtt S_n$.
\end{itemize}
\end{lemma}
\begin{proof}
	Recall from the definition \eqref{eq:def_reverse} that $H_{[t-\ell_n,t)}$ does not include the environment at time $t$, hence $W^{t,x}_{t-\ell_n}$ is measurable with respect to $\sigma(\omega_{s,y}:s\in[t-\ell_n,t),|x-y|\leq \ell_n)$. Property (i) is now clear. This measurability, together with the extra ``$+1$'' in the definition of $\mathtt T_n$, also ensures that $\{(T_k,Y_k)=(t,y)\}$ is independent of $\omega_{t,y}$, and hence $\omega_{T_k,Y_k}$ has the same law as $\omega_{0,0}$. To complete the proof of Property (ii), note that the coordinates of the sequence \eqref{eq:seq} is are defined from disjoint blocks of the environment, and hence independent. For Property~(iii), it is enough to note that the definition of $\mathtt S_n$ ensures that $\cev W^{t,x}_{t-\ell_n}$ and $\cev W^{t,y}_{t-\ell_n}$ are defined from disjoint parts of the environment for $x\neq y\in \mathtt S_n$. Thus the field $(W^{x,t}_{t-\ell_n})_{x\in\mathtt S_n}$ is independent, and in particular exchangeable.
\end{proof}

\smallskip Note that the above definitions depend on $\delta$, even though our notation does not reflect it. In the next lemma we choose $\delta$. 
\begin{lemma}\label{lem:heavy_existence}
Assume $d\geq 3$ and \eqref{eq:WD}. For every $\eps>0$ there exists $\delta>0$ such that 
\begin{align}\label{eq:PE} 
\lim_{n\to\infty}\P(E_n^c)=0.
\end{align}
\end{lemma}

\begin{proof}
For $\delta>0$ to be chosen later, let
\begin{align*}
\ell_n'\coloneqq \Big\lfloor{\frac{1+d/2-3\delta}{x\q(1+\delta)}\log n}\Big\rfloor
\end{align*}
where $x$ is the value from Proposition~\ref{prop:ldp_lower} corresponding to $\delta$. Then, for $n$ large enough,
\begin{equation}\label{eq:sdfdsf}
\begin{split}
\P\Big(\cev W_{t-\ell_n'}^{t,y}\geq n^{\frac{1+d/2-3\delta}{\q(1+\delta)}}e^{-x}\Big)
&\geq \P\Big( W_{\ell_n'}\geq e^{\ell_n'x}\Big)\\
&\geq e^{-x\q \ell_n'(1+\delta)}\\
&\geq n^{-1-d/2+3\delta}.
\end{split}
\end{equation}
Moreover, by \eqref{eq:WD} we can choose $a>0$ small enough that $\P(\inf_{m\in\N}W_m>2a)>\frac{1}{4}$, hence
\begin{equation}\label{eq:sdfdsf2}
\begin{split}
&\P\Big(\cev W_{t-\ell_n}^{t,y}\geq n^{\frac{1+d/2-3\delta}{\q(1+\delta)}}e^{-x}a\Big|\cev W_{t-\ell_n'}^{t,y}\geq n^{\frac{1+d/2-3\delta}{\q(1+\delta)}}e^{-x}\Big)\\
&\geq \P\Big(\frac{\cev W_{t-\ell_n}^{t,y}}{\cev W_{t-\ell_n'}^{t,y}}\geq a\Big|\cev W_{t-\ell_n'}^{t,y}\geq n^{\frac{1+d/2-3\delta}{\q(1+\delta)}}e^{-x}\Big)\\
&\geq 2\P(W_{\ell_n-\ell_n'}\geq 2a)-1\\
&\geq \frac 12,
\end{split}
\end{equation}
where we have used Lemma~\ref{lem:concave} in the second inequality. Combining \eqref{eq:sdfdsf} and \eqref{eq:sdfdsf2}, we have
\begin{equation}\label{eq:combined}
	\begin{split}\P\big(T_k>-\infty\big)&\geq 1-\Big(1-\frac 12n^{-1-d/2+3\delta}\Big)^{|\mathtt T_n(k)\times \mathtt S_n|}\\&\geq 1-\Big(1-\frac 12n^{-1-d/2+3\delta}\Big)^{cn^{1+d/2-2\delta}/\ell_n^{1+d/2}}\\
&\geq 1-e^{-c n^{\delta/2}}
	\end{split}
\end{equation}
If we choose $\delta>0$ such that ${\frac{1+d/2-3\delta}{\q(1+\delta)}}>\frac{2+d}{2\q}-\eps/2$, then $n^{\frac{1+d/2-3\delta}{\q(1+\delta)}}e^{-x}a\geq n^{\frac{2+d}{2\q}-\eps/2}$ holds for all $n$ large enough and thus \eqref{eq:PE} follows from \eqref{eq:combined} and the union bound.
\end{proof}

The next lemma shows that our notion of an \localization site has high probability to be visited by polymers of length $n$ starting from $\{0\}\times\Z^d$. Recall the definitions of $W^{s,x}_t$ and $\cev W^{t,x}_s$ from \eqref{eq:def_Wtx} and \eqref{eq:def_reverse}. We introduce notation similar to \eqref{eq:def_restricted} for a reverse partition function and reverse polymer measure,
\begin{align}
	\cev W_{s}^{t,x}[f(X_0/\sqrt n)]&\coloneqq {\cev E}{}^{t,x}\left[e^{\beta H_{[s,t)}(\omega,X)-(t-s)\lambda(\beta)}f(X_0/\sqrt n)\right],\label{eq:def_reverse_f}\\
	\cev \mu_{s}^{t,x}[f(X_0/\sqrt n)]&\coloneqq \frac{\cev W_s^{t,x}[f(X_0/\sqrt n)]}{\cev W_s^{t,x}}.\label{eq:bw_pm}
\end{align}

\begin{lemma}\label{lem:heavy_prop}
	Assume $d\geq 3$ and \eqref{eq:WD} and let $f$ be as in Theorem~\ref{thm:main}. There exists $(p_n)_{n\in\N}$ with $\lim_{n\to\infty}p_n=0$ such that, for all $n$ large enough and $k=1,\dots,\eeeta n^{2\delta}/2$,
\begin{align}
\P\Big(W_{n}^{T_k,Y_k}\leq n^{-\eps/4},T_k>-\infty\Big)&\leq p_n\label{eq:PEEE}
\end{align}
and for all $t\in\mathtt T_n(k)$, almost surely on $T_k=t$,
\begin{align}
\P\Big(\big|\cev W^{T_k,Y_k}_1\big[f(X_0/\sqrt n)\big]\big|\leq n^{\frac{2+d}{2\q}-\frac 58\eps}\Big|\F_{[t-\ell_n,\infty)}\Big)&\leq p_n.\label{eq:PEEEE}
\end{align}
\end{lemma}

\begin{proof}
Recall the definitions of $\eeta$, $\eeeta$, $z$ and $L$ in the beginning of Section~\ref{sec:local} and let $p\in(1,\p)$. We define an auxiliary quantity
\begin{align}\label{eq:p_n'}
	p_n'\coloneqq \sup_{t,x}\P\Big( \Big|\cev\mu_{1}^{t,x}[f(X_0/\sqrt n)]-f*\varphi(z)\Big|> \frac{|f*\varphi(z)|}{2}\Big)^{\frac{1}{2}(1-\frac{1}{p})},
\end{align}
where the supremum is over $(t,x)\in [(1-\eeeta)n,n]\times\big(zn^{1/2}+[-2\eeta n^{1/2},2\eeta n^{1/2}]\big)$, and 
\begin{equation}\label{eq:p_n}
\begin{split}
p_n\coloneqq &\sup_{t=0,\dots,n^{1-2\delta}}2\, \P\Big(W_t[\1_{X_s\in[-n^{(1-\delta)/2},n^{(1-\delta)/2}\text{ for all }s=0,..,t}]\leq n^{-\eps/8}\Big)\\
	     &+\sup_{t\in\N} 2\,\P\Big(W_t\leq \frac{4(2p_n'+n^{-\eps/8})}{|f*\varphi(z)|}\Big)+\sup_{t\in\N} 2\,\P\big(W_t\leq 2n^{-\eps/8}\big)\\
	     &+\Big(\frac{1}{2}|f*\varphi(z)|+\|f\|_\infty\Big)\sup_t \E[W_t^p]^{1/p}p_n'.
\end{split}
\end{equation}
Let us check that indeed $\lim_{n\to\infty}p_n=0$. For the first term, we estimate
\begin{align*}
&\P\big(W_t[\1_{X_s\in[-n^{(1-\delta)/2},n^{(1-\delta)/2}\text{ for all }s=0,..,t}]\leq n^{-\eps/8}\big)\\
&\leq \P(W_t\leq 2n^{-\eps/8})+\P\big(W_t[\1_{\sup_{s=0,\dots,t}|X_s|>n^{(1-\delta)/2}}]> n^{-\eps/8}\big).
\end{align*}
In the second term we apply the Markov inequality to get
\begin{align*}
\P\big(W_t[\1_{\sup_{s=0,\dots,t}|X_s|>n^{(1-\delta)/2}}]> n^{-\eps/8}\big)&\leq n^{\eps/8}\E\big[W_t[\1_{\sup_{s=0,\dots,t}|X_s|>n^{(1-\delta)/2}}]\big]\\
&=n^{\eps/8}P\big({\textstyle\sup_{s=0,\dots,t}}|X_s|>n^{(1-\delta)/2}\big)\\
&\leq n^{1-2\delta+\eps/8}\sup_{s=0,\dots,n^{1-2\delta}}P\big(|X_s|>n^{(1-\delta)/2}\big)
\end{align*}
By standard moderate deviation estimates, see \cite[Theorem~3.7.1]{DZ}, we see that the probability in the final line decays at a stretched exponential rate. 

\smallskip In view of \eqref{eq:WD}, to prove $\lim_{n\to\infty}p_n=0$ it is now enough to show $\lim_{n\to\infty}p_n'=0$. By Theorem~\ref{thmx:phase}(iii), we know that 
\begin{align*}
	&\sup_{t,x}\P\Big(\big|\cev\mu_{1}^{t,x}[f(X_0/\sqrt t+z-x/\sqrt t)]-f*\varphi(z)\big|>|f*\varphi(z)|/4\Big)\\
	&=\sup_{t\in[(1-\eeeta)n,n]}\P\Big(\big|\mu_{t-1}[f(X_t/\sqrt t+z)]-f*\varphi(z)\big|>|f*\varphi(z)|/4\Big)\\
	&\xrightarrow{n\to\infty}0.
\end{align*}
It is thus enough to show that, almost surely for all $t,x$ that appear in the supremum in \eqref{eq:p_n'},
\begin{align}\label{eq:close}
	\big|\cev\mu^{t,x}_1[f(X_0/\sqrt n)]-\cev\mu^{t,x}_1[f(X_0/\sqrt t+z-x/\sqrt t)]\big|\leq |f*\varphi(z)|/4.
\end{align}
Indeed, for all such $(t,x)$ and all $y\in\Z^d$,
\begin{align*}
	&\Big|f\Big(\frac{y}{\sqrt {n}}\Big)-f\Big(\frac{y-(x-z\sqrt {t})}{\sqrt n}\Big)\Big|\\
	&\leq\Big|f\Big(\frac{y}{\sqrt {n}}\Big)-f\Big(\frac{y}{\sqrt n}-\frac{x}{\sqrt n}+z\Big)\Big|+ |f*\varphi(z)|/8\\
									  &\leq |f*\varphi(z)|/4,
\end{align*}
where we have used \eqref{eq:def_eta2} and \eqref{eq:def_eta} in the first inequality and again \eqref{eq:def_eta} in the second inequality. Next, since $f$ vanishes outside of $[-L,L]^d$, we can use \eqref{eq:def_eta2} to get that, for all $t\in[(1-\eeeta)n,n]$ and $y\in\Z^d$,
\begin{align*}
	\big|f(y/\sqrt n)-f(y/\sqrt{t})\big|\leq |f*\varphi(z)|/8.
\end{align*}
Together with the previous display, we obtain \eqref{eq:close}.

\smallskip We continue with the proof of \eqref{eq:PEEEE}. The difficulty in this step is that in the sum in the third line of the following display, we do not know the sign of the summands. Thus, the sum may be small because all summands are small, or because they are large and cancel each other out. We have to consider both cases. On $T_k=t,Y_k=y$, we have 
\begin{align}
&	\{|\cev W_1^{t,y}[f(X_0/\sqrt n)]|\leq n^{\frac{2+d}{2\q}-\frac{5}{8}\eps} \}\notag\\
&\quad\subseteq	\Big\{\frac{\big|\cev W_1^{t,y}f(X_0/\sqrt n)]\big|}{\cev W_{t-\ell_n}^{t,y}}\leq n^{-\eps/8}\Big\}\notag\\
&\quad=\Big\{\Big|{\sum_{|x-y|\leq \ell_n}}\cev\mu_{t-\ell_n}^{t,y}(X_{t-\ell_n}=x) \cev W_1^{t-\ell_n,x}[f(X_0)/\sqrt n]\Big|\leq n^{-\eps/8}\Big\}\notag\\
&\quad\subseteq\Big\{\Big|{\sum_{|x-y|\leq \ell_n}}\cev\mu_{t-\ell_n}^{t,y}(X_{t-\ell_n}=x) \cev W_1^{t-\ell_n,x}[f(X_0)/\sqrt n]\1_{A_{t-\ell_n,x}^c}\Big|\geq p_n'\}\label{eq:T3}\\
&\qquad\cup\Big\{\Big|{\sum_{|x-y|\leq \ell_n}}\cev\mu_{t-\ell_n}^{t,y}(X_{t-\ell_n}=x) \cev W_1^{t-\ell_n,x}[f(X_0)/\sqrt n]\1_{A_{t-\ell_n,x}}\Big|\leq p_n'+n^{-\eps/8}\Big\}\notag
\end{align}
where $A_{t,x}:=\{|\cev\mu^{t,x}_1[f(X_0/\sqrt n)]-f*\varphi(z)|\leq \frac{1}{2}|f*\varphi(z)|\}$. We have used the elementary observation that, for any $a,b\in\R$ and $\eps,\delta>0$,
\begin{align}\label{eq:observation}
	|a+b|\leq \eps\qquad\implies\qquad|a|\leq\eps+\delta\text{ or }|b|\geq \delta.
\end{align}To bound the probability of the second event, note that the terms
\begin{align*}
	\cev W^{t-\ell_n,x}_1[f(X_0/\sqrt n)]\1_{A_{t-\ell_n,x}}=\cev W^{t-\ell_n,x}_1\cev\mu^{t-\ell_n,x}_1[f(X_0/\sqrt n)]\1_{A_{t-\ell_n,x}}
\end{align*}
that appear in the sum all have the same sign, so we can interchange the sum with the absolute value. Moreover, it holds that $|\cev W^{t,x}_1\1_{A_{t-\ell_n,x}}|\geq W^{t,x}_1\1_{A_{t-\ell_n,x}}|f*\varphi(z)|/2$, so we further get 
\begin{align}
&\Big\{\Big|\textstyle{\sum_{|x-y|\leq \ell_n}}\cev\mu_{t-\ell_n}^{t,y}(X_{t-\ell_n}=x) \cev W_1^{t-\ell_n,x}[f(X_0)/\sqrt n]\1_{A_{t-\ell_n,x}}\Big|\leq p_n'+n^{-\eps/8}\Big\}\notag\\
&\subseteq\Big\{\textstyle{\sum_{|x-y|\leq \ell_n}}\cev\mu_{t-\ell_n}^{t,y}(X_{t-\ell_n}=x) \cev W_1^{t-\ell_n,x}\1_{A_{t-\ell_n,x}}\leq 2(p_n'+n^{-\eps/8})/|f*\varphi(z)|\Big\}\notag\\
&\subseteq\Big\{\textstyle{\sum_{|x-y|\leq \ell_n}}\cev\mu_{t-\ell_n}^{t,y}(X_{t-\ell_n}=x) \cev W_1^{t-\ell_n,x}\1_{A_{t-\ell_n,x}^c}\geq 2p_n'/|f*\varphi(z)|\Big\}\label{eq:T2}\\
&\quad\cup\Big\{\textstyle{\sum_{|x-y|\leq \ell_n}}\cev\mu_{t-\ell_n}^{t,y}(X_{t-\ell_n}=x) \cev W_1^{t-\ell_n,x}\leq 2(2p_n'+n^{-\eps/8})/|f*\varphi(z)|\Big\}\label{eq:T1}
\end{align}
We have again used \eqref{eq:observation} for the second inclusion. Using the Markov inequality and then the H\"older inequality, we see that the probability of the event in \eqref{eq:T2} is bounded by
\begin{align*}
	\frac{|f*\varphi(z)|}{2p_n'}\sup_{t,x}\E\big[W^{t,x}_1\1_{A_{t,x}}\big]\leq \frac{|f*\varphi(z)|}{2p_n'}\sup_t \E[W_t^p]^{1/p}\sup_{t,x}\P(A_{t,x})^{1-1/p}.
\end{align*}
By definition of $p_n'$, the last quantity is equal to $\frac 12{|f*\varphi(z)|\sup_t\E[W_t^p]^{1/p}}p_n'$. Using Jensen's inequality and $|\cev W^{t,x}_1[f(X_0/\sqrt n)]|\leq \|f\|_\infty \cev W^{t,x}_1$, we can similarly show that the probability of the event in \eqref{eq:T3} is bounded by $\|f\|_\infty\sup_t\E[W_t^p]^{1/p} p_n'$. Finally, by Lemma~\ref{lem:concave}, the probability of the event in \eqref{eq:T1} is bounded by
\begin{align*}
	\sup_{t,x}2\P\Big(\cev W_1^{t,x}\leq 4(2p_n'+n^{-\eps/8})/|f*\varphi(z)|\Big).
\end{align*}

\smallskip We turn to \eqref{eq:PEEE}. Here, we have to deal with the difficulty that $W^{t,y}_{(1-\eeeta)n+2kn^{1-2\delta}}$ is, conditionally on $T_k=t,Y_k=y$, stochastically smaller than the unconditioned martingale, since we know that 
\begin{align}\label{eq:large}
\cev W_{s-\ell_n}^{s,x}\geq n^{\frac{2+d}{2\q}-\eps/2}
\end{align}
cannot holds for any $s\in (t,(1-\eeeta)n+2kn^{1-2\delta}]\cap\mathtt T_n$, $x\in\mathtt S_n$.  The bound \eqref{eq:PEEE} essentially says that the conditioning can be disregarded. Intuitively, this is because we can restrict the partition function $W^{t,y}_{(1-\eeeta)n+2kn^{1-2\delta}}$ to the area $[t,(1-\eeeta)n+2kn^{1-2\delta}]\times (y+[-n^{(1-2\delta)/2},n^{(1-2\delta)/2}]^d)$, which is much smaller than the area used to define $(T_k,Y_k)$. Therefore, starting from $(t,y)$ with high probability we will not find $(s,x)$ satisfying \eqref{eq:large} and $|x-y|\leq  n^{(1-\delta)/2}$, regardless of whether we condition on $(T_k,Y_k)=(t,y)$ or not. 

\smallskip To make this precise, we fix $k$ and write $\mathtt T_n(k)=\{t_M,\dots,t_2,t_1\}$  and $S_n=\{x_1,x_2,\dots,x_{|\mathtt S_n|}\},$ where $i\mapsto t_i$ is decreasing. The proof of \eqref{eq:PEEE} is divided into two steps, the first of which is to take care of the contribution in $[T_k,(1-\eeeta)n+2kn^{1-2\delta}]$. Since $W_{(1-\eeeta)n+2kn^{1-2\delta}}^{t_i,x_j}$ is independent of $\sigma(\omega_{s,x}:s\in [t_i-\ell_n,t_i))$, we have
\begin{align}
&\P\Big(W_{(1-\eeeta)n+2kn^{1-2\delta}}^{t_i,x_j}\leq n^{-\eps/8}\Big|T_k=t_i,Y_k=x_j\Big)\notag\\
&=\P\Big(W_{(1-\eeeta)n+2kn^{1-2\delta}}^{t_i,x_j}\leq n^{-\eps/8}\Big|\cev W_{t_a-\ell_n}^{t_a,x_b}< n^{\frac{2+d}{2\q}-\eps/2}\forall a<i,1\leq b\leq |\mathtt S_n|\Big)\notag\\
&\leq \P\Big(W_{(1-\eeeta)n+2kn^{1-2\delta}}^{t_i,x_j}[\1_{A_{t_i,x_j}}]\leq n^{-\eps/8}\Big|\cev W_{t_a-\ell_n}^{t_a,x_b}< n^{\frac{2+d}{2\q}-\eps/2}\forall a<i,1\leq b\leq |\mathtt S_n|\Big),\label{eq:conditioning}
\end{align}
where 
\begin{align*}
A_{t,x}\coloneqq \Big\{X_s\in x+[-n^{(1-\delta)/2},n^{(1-\delta)/2}]\text{ for all }s=t,\dots,(1-\eeeta)n+2kn^{1-2\delta}\Big\}.
\end{align*}
In particular, the conditioning in \eqref{eq:conditioning} is empty if $i=1$. Note that $W_{(1-\eeeta)n+2kn^{1-2\delta}}^{t_i,x_j}[\1_{A_{t_i,x_j}}]$ only depends on the environment in 
\begin{align*}
\big(t_i,(1-\eeeta)n+2kn^{1-2\delta}\times\big(x_j+[-n^{(1-\delta)/2},n^{(1-\delta)/2}]^d\big),
\end{align*}
whereas $\cev W_{t_a-\ell_n}^{t_a,x_b}$ depends on the environment in $[t_a-\ell_n,t_a)\times (x_b+[-\ell_n,\ell_n]^d)$. Consequently, we can drop all $a,b$ from the conditioning except those with $|x_b-x_j|\leq n^{(1-\delta)/2}+\ell_n$. There are at most $3dn^{(1-\delta)/2}(i-1)$ such indices $a,b$ in the conditioning in \eqref{eq:conditioning}, so by applying the definition of the conditional probability we obtain the following bound:
\begin{align*}
\P\Big(W_{(1-\eeeta)n+2kn^{1-2\delta}}^{t_i,x_j}\leq n^{-\eps/8}\Big|T_k=t_i,Y_k=x_j\Big)\leq p_n''q_n^{-3dn^{(1-\delta)/2}(i-1)},
\end{align*}
where $q_n\coloneqq \P(\cev W_{t-\ell_n}^{t,y}<n^{\frac{2+d}{2\q}-\eps/2})$ and
\begin{align*}
p_n''\coloneqq \sup_{i,j}\P(W_{(1-\eeeta)n+2kn^{1-2\delta}}^{t_i,x_j}[\1_{A_{t_i,x_j}}]\leq n^{-\eps/8}).
\end{align*}
By combining the previous bound with $\P(T_k=t_i)=(1-q_n^{|\mathtt S_n|})q_n^{|\mathtt S_n|(i-1)}$, $i=1,\dots,M$, we obtain
\begin{align*}
&\P\Big(W_{(1-\eeeta)n+2kn^{1-2\delta}}^{T_k,Y_k}\leq n^{-\eps/8},T_k>-\infty\Big)\\
&\quad \leq p_n''\sum_{i=1}^{M}\P(T_k=t_i)q_n^{3dn^{(1-\delta)/2}(i-1)}\\
&\quad \leq p_n''(1-q_n^{|\mathtt S_n|})\sum_{i=1}^{M} q_n^{|\mathtt S_n|(i-1)-3dn^{(1-\delta)/2}(i-1)}\\
&\quad \leq p_n''(1-q_n^{|\mathtt S_n|})\sum_{i=1}^{M} q_n^{|\mathtt S_n|(i-1)/2}\\
&\quad \leq 2p_n''.
\end{align*}
Now, to conclude, we can write
\begin{align*}
&\P\Big(W_{n}^{T_k,Y_k}\leq n^{-\eps/4},T_k>-\infty\Big)\leq 2p_n''+\P\Big(\frac{W_{n}^{T_k,Y_k}}{W_{(1-\eeeta)n+2kn^{1-2\delta}}^{T_k,Y_k}}\leq n^{-\eps/8},T_k>-\infty\Big)\\
&\quad \leq2p_n''+\sup_{j=1,\dots,|\mathtt S_n|} \sup_{|x-x_j|\leq 2n^{1-2\delta}}2\,\P\big(W_{n}^{(1-\eeeta)n+2kn^{1-2\delta},x}\leq 2n^{-\eps/8}\big)\\
&\quad\leq p_n,
\end{align*}
where we have again applied Lemma~\ref{lem:concave} in the second inequality.
\end{proof}

\subsection{Proof of Theorem~\ref{thm:upper_lower}(i)}\label{sec:proof_lower}

Using the results obtained so far, one can show that each \localization time has a positive probability of leading to a large jump of the martingale $(M_{n,k}^f)_{k\in\N}$ defined in \eqref{eq:def_MG}, i.e., 
\begin{align*}
\P\big(|M_{n,T_k+1}^f-M_{n,T_k}^f|>n^{-\xi(\beta)-\eps}\big)>0,
\end{align*}
from which we can conclude that $\sup_k |M_{n,k}^f|\geq n^{-\xi(\beta)-\eps}$ with high probability. However, the jumps are not independent, so even a large number of such jumps does not guarantee that the same lower bound applies to the endpoint, i.e. that  $|\X_n^f|=|M_{n,n}^f|\geq n^{-\xi(\beta)-\eps}$. 
Indeed, it is not hard to construct a martingale with large excursions but whose jumps sizes are chosen such that it is always steered back to the initial value. We refer to \cite{GPZ14} and \cite{LPS16} for a more in-depth discussion of this phenomenon, and specifically to the example from \cite[Theorem~1.2]{GPZ14}.

\smallskip As explained in Section~\ref{sec:strategy}, we overcome this complication by reducing to the simpler situation of a sum of independent centered jumps, a sufficient proportion of which is large. More precisely, we show that $M_{n,n}^f$ can be written as a sum of $n^{o(1)}$ summands that are independent conditionally on a sigma-field $\mathcal G$, each of which is larger than $n^{-\xi(\beta)-o(1)}$ with positive probability, plus a $\mathcal G$-measurable term and a negligible error-term.

\smallskip Recall that the \emph{concentration function} of a real-valued random variable $X$ is defined by 
\begin{align}\label{eq:def_Q}
Q_X(\lambda)\coloneqq {\textstyle \sup_{x\in\R}}\P(X\in [x,x+\lambda]).
\end{align}
The proof of Theorem~\ref{thm:upper_lower}(i) is based on the following result. 
\begin{thmx}[{\cite[Theorem~1]{R61}}]\label{thmx:rogozin}
There exists $\newconstant{\label{c:rogozin}}>0$ such that, for all $N\in\N$, all $\lambda>0$ and all independent, real random variables $X_1,\dots,X_N$,
\begin{align}\label{eq:rogozin}
Q_{\sum_{i=1}^NX_i}(\lambda)\leq \oldconstant{c:rogozin}\Big(\textstyle\sum_{i=1}^N(1-Q_{X_i}(\lambda))\Big)^{-1/2}.
\end{align}
\end{thmx}

In our application the summands are only conditionally independent, so we introduce the conditional concentration function $Q_{X|\mathcal G}(\lambda)$, which is defined as in \eqref{eq:def_Q} with $\P(\cdot)$ replaced by $\P(\cdot|\mathcal G)$. Note that $Q_{X|\mathcal G}(\lambda)$ is random. The following result proves the desired decomposition.

\begin{proposition}\label{prop:decomposition}
Assume $d\geq 3$, \eqref{eq:WD} and let $f$ be as in Theorem~\ref{thm:upper_lower}. For every $\eps>0$ there exists $\newconstant{\label{c:decomp}}\in(0,1)$ and a sequence $(N_n)_{n\in\N}$ with $\lim_{n\to\infty}N_n=\infty$ such that the following hold for all $n\in\N$: there exist a sigma-field $\mathcal G$, an event $\widehat E_n$ and random variables $A$, $B$, $Z_1,\dots,Z_{N_n }$ such that, on $\widehat E_n$,
\begin{align}\label{eq:decomposition}
n^{-d/2}\sum_{x\in \Z}f(x/\sqrt n)W_n^{0,x}= A+B+\sum_{i=1}^{N_n }Z_i,
\end{align}
such that
\begin{itemize}
 \item $A$ and $\widehat E_n$ are $\mathcal G$-measurable,
 \item $Z_1,\dots,Z_{N_n}$ are independent under $\P(\cdot|\mathcal G$), on $\widehat E_n$,
\end{itemize}
and such that
\begin{align}
\lim_{n\to\infty}\P\big(\widehat E^c_n\big)&=0,\label{eq:PEhat}\\
\lim_{n\to\infty}\P\big(|B|>n^{-d},\widehat E_n\big)&=0\label{eq:PB},\\
\lim_{n\to\infty}\sup_{i=1,\dots,N_n}\P\Big(Q_{Z_i|\mathcal G}(4n^{-\xi(\beta)-\eps})\geq 1-\oldconstant{c:decomp},\widehat E_n\Big)N_n&=0.\label{eq:PZ}
\end{align}
\end{proposition}

Let us first see how Proposition~\ref{prop:decomposition} yields the desired conclusion:

\begin{proof}[Proof of Theorem~\ref{thm:upper_lower}(i) using Proposition~\ref{prop:decomposition}]
We have
\begin{align*}
	&\P(|M_{n,n}|\leq n^{-\xi(\beta)-\eps})\\
	&\leq P(\widehat E^c_n)+\P\Big(A+B+{\textstyle \sum_{i=1}^{N_n} Z_i}\in A'+[-n^{-\xi(\beta)-\eps},n^{-\xi(\beta)-\eps}],\widehat E_n\Big)\\
&\leq P(\widehat E^c_n)+P\big(|B|\geq n^{-\xi(\beta)-\eps},\widehat E_n\big)\\
&\qquad +\P\Big(A+{\textstyle \sum_{i=1}^{N_n} Z_i}\in A'+[-2n^{-\xi(\beta)-\eps},2n^{-\xi(\beta)-\eps}],\widehat E_n\Big),
\end{align*}
where $A'=n^{-d/2}\sum_x f(x/\sqrt n)$. The first two terms converge to zero by \eqref{eq:PEhat} and \eqref{eq:PB}, together with the fact that $\xi(\beta)<d$. For the third term, we write, for any $a\in(0,1)$,
\begin{align*}
&\P\Big(A+{\textstyle \sum_{i=1}^{N_n} Z_i}\in A'+[-2n^{-\xi(\beta)-\eps},2n^{-\xi(\beta)-\eps}],\widehat E_n\Big)\\
&=\E\Big[\P\Big(A+{\textstyle \sum_{i=1}^{N_n} Z_i}\in A'+[-2n^{-\xi(\beta)-\eps},2n^{-\xi(\beta)-\eps}]\Big|\mathcal G\Big)\1_{\widehat E_n}\Big]\\
&\leq \E\Big[Q_{A+{ \sum_{i=1}^{N_n} Z_i}|\mathcal G}(4n^{-\xi(\beta)-\eps})\1_{\widehat E_n}\Big]\\
&=\E\Big[Q_{{ \sum_{i=1}^{N_n} Z_i}|\mathcal G}(4n^{-\xi(\beta)-\eps})\1_{\widehat E_n}\Big]\\
&\leq  a+\P\Big(Q_{\textstyle \sum_{i=1}^{N_n} Z_i|\mathcal G}(4n^{-\xi(\beta)-\eps})>a,\widehat E_n\Big).
\end{align*}
The two equalities are due to the fact that $\widehat E_n$ and $A$ are $\mathcal G$-measurable and the first inequality follows directly from the definition of the concentration function $Q$. Since $Z_1,\dots,Z_{N_n}$ are conditionally independent, we can apply Theorem~\ref{thmx:rogozin} to bound the second term by
\begin{align*}
&\P\Big({\textstyle \sum_{i=1}^{N_n}}\big(1-Q_{Z_i|\mathcal G}(4n^{-\xi(\beta)-\eps})\big)<\frac{\oldconstant{c:rogozin}^2}{a^2},\widehat E_n\Big)\\
&\leq \P\Big(\exists i\in\{1,\dots,{N_n}\}\colon Q_{Z_i|\mathcal G}(4n^{-\xi(\beta)-\eps})>1-\frac{\oldconstant{c:rogozin}^2}{{N_n}a^2},\widehat E_n\Big)\\
&\leq \sum_{i=1}^{N_n}\P\Big(Q_{Z_i|\mathcal G}(4n^{-\xi(\beta)-\eps})>1-\frac{\oldconstant{c:rogozin}^2}{{N_n}a^2},\widehat E_n\Big)\\
&\leq N_n\sup_{i=1,\dots,N_n}\P\Big(Q_{Z_i|\mathcal G}(4n^{-\xi(\beta)-\eps})>1-\frac{\oldconstant{c:rogozin}^2}{{N_n}a^2},\widehat E_n\Big).
\end{align*}
The claim follows from \eqref{eq:PZ} by choosing $a\coloneqq {N_n}^{-1/4}$.
\end{proof}

\begin{proof}[Proof of Proposition~\ref{prop:decomposition}]
Recall the construction from the beginning of Section~\ref{sec:local}. Let $\delta$ be as in Lemma~\ref{lem:heavy_existence}, $(p_n)_{n\in\N}$ as in Lemma~\ref{lem:heavy_prop} and set 
\begin{align*}
	N_n&\coloneqq \big\lfloor p_n^{-1/2}\wedge n^{\delta/6}\big\rfloor,\\
\widehat E_n&\coloneqq \Big\{T_k>-\infty\text{ for all }k=1,\dots,N_n\Big\}.
\end{align*}
Now \eqref{eq:PEhat} follows from Lemma~\ref{lem:heavy_existence} and $\widehat E_n\subseteq E_n$. Let $\mathcal I\coloneqq \{(T_1,Y_1),\dots,(T_{N_n},Y_{N_n})\}$ and 
\begin{align*}
\mathcal G\coloneqq \sigma\big(T_k,Y_k:k=1,\dots,N_n,\,\omega_{t,x}:(t,x)\notin \mathcal I\big).
\end{align*}
Now $\widehat E_n$ is $\mathcal G$ measurable by definition. To define the decomposition \eqref{eq:decomposition}, we introduce a truncated path-energy, which disregards the environment in $\mathcal I$,
\begin{align}\label{eq:def_HI}
H^{\beta}_{[s,t],\mathcal I^c}(\omega,X)&\coloneqq {\textstyle \sum_{i=s}^t}\big(\beta\omega_{i,X_i}-\lambda(\beta)\big)\1_{(i,X_i)\notin \mathcal I},
\end{align}
We also consider, for $K\subseteq\{1,\dots,{N_n}\}$, the event that a path visits all sites $\{(T_k,Y_k),k\in K\}$,
\begin{align}\label{eq:Cs}
\Visit(K)\coloneqq \{X_{T_k}=Y_k\text{ for all }k\in K\},
\end{align}
and similarly, for $K, K'\subseteq\{1,\dots,{N_n}\}$,
\begin{align}\label{eq:CCs}
\Visit(K,K')\coloneqq \{X_{T_k}=Y_k\text{ for all }k\in K, X_{T_k}\neq Y_k\text{ for all }k\in K'\}.
\end{align}
On $\widehat E_n$, we decompose according to which of the sites $(T_1,Y_1),\dots,(T_{N_n},Y_{N_n})$ are visited,
\begin{align*}
W_n^{0,x}&=\sum_{K\subseteq \{1,\dots,{N_n}\}}E^{0,x}[e^{H_n(\omega,X)-n\lambda(\beta)}\1_{\Visit(K,K^c)}]\\
&=\sum_{K\subseteq \{1,\dots,{N_n}\}}E^{0,x}\Big[e^{H_{[1,n],\mathcal I^c}^\beta(\omega,X)}\prod_{k\in K}e^{\beta\omega_{T_k,Y_k}-\lambda(\beta)}\1_{\Visit(K,K^c)}\Big]\\
&=\sum_{K\subseteq \{1,\dots,{N_n}\}}E^{0,x}\Big[e^{H_{[1,n],\mathcal I^c}^\beta(\omega,X)}\prod_{k\in K}\big(e^{\beta\omega_{T_k,Y_k}-\lambda(\beta)}-1+1\big)\1_{\Visit(K,K^c)}\Big]\\
&=\sum_{K\subseteq \{1,\dots,{N_n}\}}\sum_{K'\subseteq K}E^{0,x}\Big[e^{H_{[1,n],\mathcal I^c}^\beta(\omega,X)}\prod_{k\in K'}\big(e^{\beta\omega_{T_k,Y_k}-\lambda(\beta)}-1\big)\1_{\Visit(K,K^c)}\Big]\\
&=\sum_{K'\subseteq \{1,\dots,{N_n}\}}E^{0,x}\Big[e^{H_{[1,n],\mathcal I^c}^\beta(\omega,X)}\prod_{k\in K'}\big(e^{\beta\omega_{T_k,Y_k}-\lambda(\beta)}-1\big )\1_{\Visit(K')}\Big]\\
&=: \sum_{K'\subseteq \{1,\dots,{N_n}\}}W_n^{0,x}(K').
\end{align*}
Now we define, on $\widehat E_n$,
\begin{align}
	A&\coloneqq n^{-d/2}\sum_{x\in\Z^d}f(x/\sqrt n)W_n^{0,x}(\varnothing),\label{eq:def_A}\\
B&\coloneqq n^{-d/2}\sum_{x\in\Z^d}f(x/\sqrt n)\sum_{K\subseteq \{1,\dots,{N_n}\},|K|\geq 2}W_n^{0,x}(K),\label{eq:def_B}\\
Z_k&\coloneqq n^{-d/2}\sum_{x\in\Z^d}f(x/\sqrt n)W_n^{0,x}(\{k\}).\label{eq:def_Z}
\end{align}
Since $(T_1,Y_1),\dots,(T_{N_n},Y_{N_N})$ are $\mathcal G$-measurable and $W_n^{0,x}(\varnothing)$ does not depend on $\omega_{T_1,Y_1},\\\dots,\omega_{T_{N_n},Y_{N_n}}$, we see that $A$ is  $\mathcal G$-measurable. Moreover, on $\widehat E_n$, $(T_1,Y_1),\dots,(T_{N_n},Y_{N_n})$ are all distinct and therefore $W_n^{0,x}(K_1)$ and $W_n^{0,x}(K_2)$ are independent whenever $K_1\cap K_2=\varnothing$, conditionally on $\mathcal G$. In particular, $Z_1,\dots,Z_{N_n}$ are conditionally independent. 

It remains to show that \eqref{eq:PB} and \eqref{eq:PZ} are satisfied, which we prove below.
\end{proof}

\begin{proof}[Proof of \eqref{eq:PZ}]
On $\widehat E_n$, we have
\begin{align*}
	Z_k=\widehat Z_k (e^{\beta\omega_{T_k,Y_k}-\lambda(\beta)}-1),
\end{align*}
where 
\begin{align*}
\widehat Z_k&=n^{-d/2}\sum_{x\in\Z^d}f(x/\sqrt n)E^{0,x}\big[e^{\widehat H_{n,\mathcal I^c}^\beta(\omega,X)}\1_{X_{T_k}=Y_k}\big].
\end{align*}
Since $e^{\omega_{t,y}-\lambda(\beta)}-1$ is not constant and has expectation zero, we find $c,c'>0$ such that 
\begin{align*}
\P\big(e^{\beta\omega_{t,y}-\lambda(\beta)}-1>c\big)\wedge \P\big(e^{\beta\omega_{t,y}-\lambda(\beta)}-1<-c\big)>c'.
\end{align*}
Moreover, $\widehat Z_k$ is $\mathcal G$-measurable and $\omega_{T_k,Y_k}$ has law $\P$ and is independent of $\mathcal G$, therefore $Q_{Z_k|\mathcal G}\big(2c\widehat Z_k\big)<1-c'$
and thus 
\begin{align*}
\P\Big(Q_{Z_k|\mathcal G}\big(4n^{-\xi(\beta)-\eps}\big)\geq 1-c',\widehat E_n\Big)\leq \P\Big(\widehat Z_k\leq 2n^{-\xi(\beta)-\eps}/c,\widehat E_n\Big)
\end{align*}
To show get a lower bound for $\widehat{Z}_k$, we compare $\widehat{Z}_k$ to  $\cev W^{T_k,Y_k}[f(X_0/\sqrt n)]W^{T_k,Y_k}$, which amounts to undoing the truncation from \eqref{eq:def_HI}, except for the site $(T_k,Y_k)$. To justify this, we first observe that the truncation $\mathcal I^c$ in $\widehat{Z}_k$ can be ignored if $\mathcal I$ is visited only at $(T_k,Y_k)$, i.e.,  
\begin{align*}
	e^{H^\beta_{n,\mathcal I^c}(\omega,\pi)}=	e^{\beta H_{[1,n]\setminus \{k\}}(\omega,\pi)-(n-1)\lambda(\beta)}
\end{align*}
for all paths $\pi\in V(\{k\},\{1,\dots,N_n\}\setminus\{k\})$, where we recall the definition of the event $V(K,K')$ in \eqref{eq:CCs} and the notation \eqref{eq:def_H}. Hence, we observe that
\begin{align*}
&	\left|\widehat{Z}_k-n^{-d/2}\cev W^{T_k,Y_k}[f(X_0/\sqrt n)]W^{T_k,Y_k}\right|\\
&=n^{-d/2}\Big|\sum_x f(x/\sqrt n)E^{0,x}\Big[\Big(e^{H^\beta_{n,\mathcal I^c}(\omega,X)}-e^{\beta H_{[1,n]\setminus\{k\}}(\omega,X)-(n-1)\lambda(\beta)}\Big)\1_{X_{T_k}=Y_k}\Big]\Big|\\
&\leq  \overline B_k,
\end{align*}
where the error-term $\overline B_k$ is defined by
\begin{equation}\label{eq:def_Bbar}
	\begin{split}&\overline B_k\coloneqq n^{-d/2}\|f\|_\infty\sum_{K\subseteq\{1,\dots,N_n\}:|K|>1,k\in K}\\
&\qquad\Big|\sum_{x\in[-Ln^{1/2},Ln^{1/2}]^d}E^{0,x}\Big[e^{H^\beta_{[1,n],\mathcal I^c}(\omega,X)}\Big(\prod_{l\in K\setminus\{k\}}e^{\beta\omega_{T_l,Y_l}-\lambda(\beta)}-1\Big)\1_{V(K)}]\Big|,
	\end{split}\end{equation}
	where $L$ is as defined in the beginning of Section \ref{sec:local}. This term is similar to $B$ and in the process of proving \eqref{eq:PB} we will also prove that there exists $c>0$ such that, for all $n\in\N$ and $k=1,\dots,N_n$,
\begin{align}\label{eq:PBB}\tag{\ref{eq:PB}'}
	\P\big(\overline B_k>n^{-d},\widehat{E}_n\big)\leq cn^{-\delta/3}.
\end{align}
Assuming \eqref{eq:PBB}, we can now conclude: for any $k=1,\dots,N_n$,
\begin{align*}
&	\P\Big(Q_{Z_k|\mathcal G}\big(4n^{-\xi(\beta)-\eps}\big)\geq 1-c',\widehat E_n\Big)\\
&\leq \P\Big(\widehat Z_k\leq 2n^{-\xi(\beta)-\eps}/c,\widehat E_n\Big)\\
& \leq \P\big(B_k'>n^{-d},\widehat{E}_n\big)+\P\Big(n^{-d/2}\big|\cev W^{T_k,Y_k}_1[f(X_0/\sqrt n)]\big|W^{T_k,Y_k}_n\leq 3n^{-\xi(\beta)-\eps}/c,\widehat{E}_n\Big)\\
&\leq \P\big(B_k'>n^{-d},\widehat{E}_n\big)+\P\Big(n^{-d/2}\big|\cev W^{T_k,Y_k}_1[f(X_0/\sqrt n)]\big|\leq 3n^{-\xi(\beta)-\eps/2}/c,\widehat{E}_n\Big)\\
&\quad+\P\Big(W^{T_k,Y_k}_n\leq n^{-\eps/2},\widehat{E}_n\Big)\\
&\leq 2p_n+cn^{-\delta/3}\\
&\leq c'N_n^{-2},
\end{align*}
where the last inequality follows from the definition of $N_n$.
\end{proof}

\begin{proof}[Proof of \eqref{eq:PB} and \eqref{eq:PBB}]
	We start with \eqref{eq:PB} and then described the modifications necessary for \eqref{eq:PBB}. The idea is that, by construction, the spatial distance between $Y_k$ and $Y_l$ is typically $\geq n^{1/2(1-\delta)}$ whereas the separation in time satisfies 
	\begin{align*}
		|T_k-T_l|\leq T_{N_n}-T_1\leq n^{1-2\delta+\delta/6}\ll n^{1-\delta}.
	\end{align*}
 Hence the probability for simple random walk to visit more than one such area decays stretched exponentially. This rapid decays dominates any gains from visiting favorable areas, which are of polynomial order. 

\smallskip To make this precise, we introduce two events that ensure that the spatial separation of $(Y_k)_{k=1,\dots,N_n}$ as well as the gains from the environment are typical,
\begin{align*}
	F_n^1&\coloneqq \Big\{\cev W_{t-\ell_n}^{t,y}\leq n^{2+d/2}\text{ for all }(t,y)\in \mathtt \mathtt T_n\times\big(zn^{1/2}+[-2\eeta n^{1/2},2\eeta n^{1/2}]^d\big)\Big\},\\
F_n^2&\coloneqq \Big\{|Y_k-Y_l|\geq n^{1/2(1-\delta)}\text{ for all }k,l=1,\dots,N_n,k\neq l\Big\}.
\end{align*}
Since $\E[\cev W_{t-\ell_n}^{t,y}]=1$, we can use Markov inequality and the union bound to get
\begin{align*}
\P((F_n^1)^c)\leq \sum_{(t,y)\in\mathtt \mathtt T_n\times(zn^{1/2}+[-2\eeta n^{1/2},2\eeta n^{1/2}]^d)}\P\Big(\cev W_{t-\ell_n}^{t,y}>n^{2+d/2}\Big)\leq 4n^{-1}.
\end{align*}
Moreover, by Lemma~\ref{lem:properties}(ii),  conditionally on $\widehat E_n$, $(Y_k)_{k=1,\dots,N_n}$ are i.i.d. and uniformly distributed on $\mathtt S_n$. We thus have
\begin{equation*}
\begin{split}
\P\big((F_n^2)^c\big|\widehat E_n\big)&\leq \sum_{k,l=1,\dots,N_n,k\neq l}\P\big(|Y_k-Y_l|\leq n^{1/2(1-\delta)}\big|\widehat E_n\big)\\
&\leq \binom{\lfloor n^{\delta/6}\rfloor }{2} \frac{\big|\Z^d\cap[-n^{1/2(1-\delta)},n^{1/2(1-\delta)}]^d\big|}{|\mathtt S_n|}\\
&\leq cn^{\delta/3-d\delta/2}\ell_n^d\\
&\leq c n^{-\delta/3}.
\end{split}
\end{equation*}
On $\widehat E_n\cap F_n^1\cap F_n^2$, we will now bound the first moment of $B$. More precisely, we estimate
\begin{align*}
\P\big(|B|\geq n^{-d},\widehat E_n\big)&\leq \P\big((F_n^1\cap F_n^2)^c,\widehat E_n\big)+\P\big(|B|\geq n^{-d},\widehat E_n\cap F_n^1\cap F_n^2\big)
\end{align*}
and by the Markov inequality the second term is bounded by
\begin{equation}\label{eq:markovv}
	\begin{split}
&n^d\E\big[|B|\1_{\widehat E_n\cap F_n^1\cap F_n^2}\big] \\
&\qquad \leq \|f\|_\infty n^{d/2}\sum_{x\in[-Ln^{1/2},Ln^{1/2}]^d}\sum_{K'\subseteq\{1,\dots,{N_n}\},|K'|\geq 2}\E\big[\1_{\widehat E_n\cap F_n^1\cap F_n^2}|W_n^{0,x}(K')|\big],
	\end{split}
	\end{equation}
where $L$ is chosen large enough that the support of $f$ is contained in the interior of $[-L,L]^d$. To bound the last expectation, we have to consider not just how often a path visits $\mathcal I$, but also how often it comes close to $\mathcal I$. We introduce the events $\VVisit(K)$ and $\VVisit(K,K')$, which are defined as in \eqref{eq:Cs} and \eqref{eq:CCs} with ``$X_{T_k}=Y_k$'' replaced by ``$|X_{T_k}-Y_k|\leq \ell_n$'' and ``$X_{T_k}\neq Y_k$'' replaced by ``$|X_{T_k}-Y_k|> \ell_n$''. For $x\in\Z^d$,
\begin{align}
&\E\big[\1_{\widehat E_n\cap F_n^1\cap F_n^2}|W_n^{0,x}(K')|\big]\notag\\
&\leq \E\Big[\1_{\widehat E_n\cap F_n^1\cap F_n^2}E^{0,x}\Big[e^{H_{[1,n],\mathcal I^c}^\beta(\omega,X)}\prod_{k\in K'}\big|e^{\beta\omega_{T_k,X_{T_k}}-\lambda(\beta)}-1\big|\1_{\Visit(K')}\Big]\Big]\notag\\
&\leq \E\Big[\1_{\widehat E_n\cap F_n^1\cap F_n^2}E^{0,x}\Big[e^{H_{[1,n],\mathcal I^c}^\beta(\omega,X)}\prod_{k\in K'}\big|e^{\beta\omega_{T_k,X_{T_k}}-\lambda(\beta)}-1\big|\1_{\VVisit(K')}\Big]\Big]\notag\\
&= \sum_{K'\subseteq K}\E\Big[\1_{\widehat E_n\cap F_n^1\cap F_n^2}E^{0,x}\Big[e^{H_{[1,n],\mathcal I^c}^\beta(\omega,X)}\prod_{k\in K'}\big|e^{\beta\omega_{T_k,X_{T_k}}-\lambda(\beta)}-1\big|\1_{\VVisit(K,K^c)}\Big]\Big]\label{eq:notag}\\
&\leq \sum_{K'\subseteq K}(c')^{|K|}\E\Big[\1_{\widehat E_n\cap F_n^1\cap F_n^2}E^{0,x}\Big[e^{H_{[1,n],\mathcal I^c}^\beta(\omega,X)}\1_{\VVisit(K,K^c)}\Big]\Big]\notag,
\end{align}
where $c'\coloneqq \E[|e^{\beta\omega_{T_k,X_{T_k}}-\lambda(\beta)}-1|\vee 1]$. In the final line, we used that $\omega_{T_1,Y_1},\dots,\omega_{T_{N_n},Y_{N_n}}$ are independent of $\widehat{E}_n$, $F_n^1$ and $F_n^2$ (recall Lemma~\ref{lem:properties}(i)). Taking now a sum over $K'$ and using that there are at most $2^{|K|}$ subsets of $K$, we obtain 
\begin{align*}
	&\sum_{K'\subseteq \{1,\dots,N_n\},|K'|\geq 2}\E\big[\1_{\widehat E_n\cap F_n^1\cap F_n^2}|W_n^{0,x}(K')|\big]\\
	&\quad\leq \sum_{K\subseteq\{1,\dots,N_n\}, |K|\geq 2} (2c')^{|K|} \E\Big[\1_{\widehat E_n\cap F_n^1\cap F_n^2}E^{0,x}\Big[e^{H_{[1,n],\mathcal I^c}^\beta}\1_{\VVisit(K,K^c)}\Big]\Big].
\end{align*}
Next, we integrate out the environment in $([1,n]\setminus\bigcup_{k\in K}[T_k-\ell_n,T_k))\times\Z^d$, which yields
\begin{equation}\label{eq:detail}
	\begin{split}
&\E\Big[\1_{\widehat E_n\cap F_n^1\cap F_n^2}E^{0,x}\Big[e^{H_{[1,n],\mathcal I^c}^\beta}\1_{\VVisit(K,K^c)}\Big]\Big]\\
&\quad\leq \E\Big[\1_{\widehat E_n\cap F_n^1\cap F_n^2}E^{0,x}\Big[\prod_{k\in K}e^{\beta H_{[T_k-\ell_n,T_k)}-\ell_n \lambda(\beta)}\1_{\VVisit(K,K^c)}\Big]\Big].
	\end{split}
\end{equation}
This step will be justified in detail below. Now, we have $\ell_n=\lfloor \log^2(n)\rfloor$ and, on $\widehat E_n\cap F_n^2$ for $k,l=1,\dots,N_n,k\neq l$,
\begin{align}
T_k&\in[(1-\eeeta)n,n],\\
|T_{l}-T_k|&\in [n^{1-2\delta},2n^{1-2\delta+\delta/6}],\label{eq:b1}\\
|Y_{l}-Y_{k}|&\in [n^{1/2(1-\delta)},n^{1/2}].\label{eq:b2}
\end{align}
Thus, by the local central limit theorem for the simple random walk, for  $k<l\in\{1,\dots,{N_n}\}$,
\begin{align*}
\sup_{|x-Y_k|\leq \ell_n,|z-Y_{l}|\leq \ell_n,|y-z|\leq \ell_n}\frac{P(X_{T_{l}-\ell_n}=y|X_{T_k}=x,X_{T_{l}}=z)}{P\big(X_{T_{l}-\ell_n}=y\big|X_{T_{l}}=z\big)}&\leq c'',\\
\sup_{|x|\leq Ln^{1/2},|z-Y_{k}|\leq \ell_n,|y-z|\leq \ell_n}\frac{P(X_{T_{k}-\ell_n}=y|X_{0}=x,X_{T_{k}}=z)}{P\big(X_{T_k-\ell_n}=y\big|X_{T_{k}}=z\big)}&\leq c''.
\end{align*}
Hence, on $\widehat E_n\cap F_n^1\cap F_n^2$, we can further bound the integrand in \eqref{eq:detail} by 
\begin{align*}
&E^{0,x}\Big[{\textstyle \prod_{k\in K}}e^{\beta H_{[T_k-\ell_n,T_k)}-\ell_n\lambda(\beta)}\1_{\VVisit(K)}\Big]\\
&\quad=\sum_{(x_k,x_k')_{k\in K}:|Y_k-x_k|\leq \ell_n}E^{0,x}\Big[{\textstyle \prod_{k\in K}}e^{\beta H_{[T_k-\ell_n,T_k)}-\ell_n\lambda(\beta)}\1_{X_{T_k}=x_k, X_{T_k-\ell_n}=x_k'}\Big]\\
&\quad\leq (c'')^{|K|}\sum_{(x_k)_{k\in K}:|Y_k-x_k|\leq \ell_n} P^{0,x}(X_{T_k}=x_k\text{ for all } k\in K) \prod_{k=1}^{N_n}\cev W_{T_k-\ell_n}^{T_k,x_k}\\
&\quad\leq (c'')^{|K|} P^{0,x}(\VVisit_{K}) n^{|K|(2+d/2)}\\
&\quad\leq (c''')^{|K|} \ell_n^{|K|d}n^{|K|(2+d/2)}\prod_{i=2}^{|K|}n^{-d/2(T_{k_i}-T_{k_{i-1}})}e^{-c'''\frac{(Y_{k_i}-Y_{k_{i-1}})^2}{T_{k_i}-T_{k_{i-1}}}}\\
&\quad \leq \big(c''' n^{3+d/2} \big)^{|K|}e^{-c''' (|K|-1)n^{\frac 56\delta}}
\end{align*}
where we write $K=\{k_1,\dots,k_{|K|}\}$. The second inequality is due to the local central limit theorem, the first inequality uses the definition of $F_n^1$, and the final inequality uses the bounds \eqref{eq:b1}--\eqref{eq:b2}. Note that there are at most $n^r$ sets $K\subseteq \{1,\dots,N_n\}$ of cardinality $r$, so obtain, for every $x\in[-Ln^{1/2},Ln^{1/2}]^d$,
\begin{align*}
\sum_{K'\subseteq\{1,\dots,{N_n}\},|K'|\geq 2}\E\big[\1_{\widehat E_n\cap F_n^1\cap F_n^2}|W_n^{0,x}(K')|\big]&\leq \sum_{r=2}^{N_n}\big(c''' n^{4 +d/2}\big)^{r}e^{-c''' (r-1)n^{\frac 56\delta}}\\
&\leq 2(c''' n^{4+d/2})^2e^{-c''' n^{\frac 56\delta}}.
\end{align*}
The final line is bounded by $n^{-d/2-1}$ for $n$ large enough, so \eqref{eq:PB} follows from \eqref{eq:markovv}. It remains to prove \eqref{eq:PBB}. We repeat the arguments leading up to \eqref{eq:detail} with $|B|$ replaced by $\overline B_k$. The main difference is that $\prod_{l\in K'}|e^{\beta \omega_{T_l,Y_l}-\lambda(\beta)}-1|$ in \eqref{eq:notag} is replaced by $|\prod_{l\in K'\setminus\{k\}} e^{\beta\omega_{T_l,Y_l}-\lambda(\beta)}-1|$. We thus obtain a similar bound,
\begin{align*}
&	n^{d}	\E[\1_{\widehat{E}_n\cap F_n^1\cap F_n^2}\overline B_k]\\
	&\quad\leq \|f\|_\infty\sum_{|x|\leq Ln^{1/2}}\sum_{K\subseteq \{1,\dots,N_n\},|K|\geq 2} (\overline c')^{|K|}\E[\1_{\widehat{E}_n\cap F_n^1\cap F_n^2}E^{0,x}[e^{H^\beta_{[1,n],\mathcal I^c}}\1_{\VVisit(K,K^c)}] ],
\end{align*}
where $\overline c':=\E[e^{\beta\omega_{0,0}-\lambda(\beta)}\vee 1]$. By the same argument as above we obtain, for  $n$ large enough,
\begin{align*}
	\P(B_k>n^{-d},\widehat{E}_n)\leq 5 n^{-1}+cn^{-\delta/3}.\tag*{\qedhere}
\end{align*}
\end{proof}

\begin{proof}[Proof of \eqref{eq:detail}]
	To simplify the notation, we will use bold symbols, e.g., $\boldsymbol t$ and $\boldsymbol T$, for vector-valued quantities $t_1,...,t_{N_n}$ and $T_1,...,T_{N_n}$. By interchanging the order of integration in \eqref{eq:detail}, we see that it is enough to show that for all $\mathbf t=(t_1,\dots,t_{N_n})$ and $\mathbf y=(y_1,\dots,y_{N_n})$ such that $\big\{(\mathbf T,\mathbf Y)=(\mathbf y,\mathbf y)\big\}\subseteq F_n^2\cap \widehat E_n$ and all paths $\pi\in \VVisit(K,K^c)$, it holds that 
\begin{align*}
		\E\Big[\1_{(\mathbf T,\mathbf Y)=(\mathbf t,\mathbf y),F_n^1}e^{H^\beta_{[1,n],\mathcal I^c}(\omega,\pi)}\Big|\mathcal F_{\boldsymbol t}\Big]\leq \E\big[\1_{(\mathbf T,\mathbf Y)=(\mathbf t,\mathbf y),F_n^1}\big|\mathcal F_{\boldsymbol t}\big]\prod_{k\in K}e^{\beta H_{[t_k-\ell_n,t_k]}(\omega,\pi)-\ell_n\lambda(\beta)},
	\end{align*}
	where $\F_{\boldsymbol t}:=\sigma(\omega_{t,x}:t\in\bigcup_{k\in K}[t_k-\ell_n,t_k])$. Recall that the \localization sites $Y_1,\dots,Y_{N_n}$ have been defined with the help of an auxiliary sequence of random variables $U_1,\dots,U_{N_n}$ in the beginning of Section~\ref{sec:local}. Integrating out this randomness gives
\begin{align*}
&	\E[\1_{(\mathbf T,\mathbf Y)=(\mathbf t,\mathbf y)}|\sigma(\omega)]=\prod_{k=1}^{N_n}\frac{ \1{\{\cev W^{t,x}\leq n^{\frac{2+d}{2\q}-\eps/2}\text{ for all }t\in\mathtt T_n\cap (t_k,2kn^{1-2\delta}],x\in \mathtt S_n\}}}{\big|\{x\in\mathtt S_n:\cev W^{t_k,x}>n^{\frac{2+d}{2\q}-\eps/2}\}\big|}.
\end{align*}
By assumption, $\pi$ does not visit $\bigcup_{k\in K^c}[t_k-\ell_n,t_k]\times[y_k-\ell_n,y_k+\ell_n]^d$, therefore this expression is a non-increasing function of $(\omega_{t,\pi_t})_{t\in [1,n]\setminus \bigcup_{k\in K}[t_k-\ell_n,t_k]}$. In addition, $\1_{F_n^1}$ is non-increasing in all coordinates. Thus $\E\big[\1_{(\mathbf T,\mathbf Y)=(\mathbf t,\mathbf y),F_n^1}\big|\mathcal F'_{\boldsymbol t}\big]$ is also a non-increasing function of $(\omega_{t,\pi_t})_{t\in [1,n]\setminus \bigcup_{k\in K}[t_k-\ell_n,t_k]}$, where $\F'_{\boldsymbol t}\coloneqq \sigma(\omega_{t,x}:t\in\bigcup_{k\in K}[t_k-\ell_n,t_k]\text{ or } x=\pi_t)$.

\smallskip On the other hand, it is clear that $e^{H^\beta_{[1,n],\mathcal J^c}(\omega,\pi)}$ is $\F_{\boldsymbol t}'$-measurable and non-decreasing in $(\omega_{t,\pi_t})_{t\in[1,n]\setminus \bigcup_{k\in K}[t_k-\ell_n,t_k]}$. The claim follows from the FKG inequality, \cite[Theorem 3]{P74}:

\begin{align*}
	&\E\big[ e^{H^\beta_{[1,n],\mathcal I^c}(\omega,\pi)}\1_{(\mathbf T,\mathbf Y)=(\mathbf t,\mathbf y),F_n^1}\big|\F_{\boldsymbol t}\big]\\
	&= \E\Big[ e^{H^\beta_{[1,n],\mathcal I^c}(\omega,\pi)}\E\big[\1_{(\mathbf T,\mathbf Y)=(\mathbf t,\mathbf y),F_n^1}\big|\F'_{\boldsymbol t}\big]\Big|\F_{\boldsymbol t}\Big]\\
	&\leq  \E\big[ e^{H^\beta_{[1,n],\mathcal I^c}(\omega,\pi)}\big|\F_{\boldsymbol t}\big]\E\Big[\E\big[\1_{(\mathbf T,\mathbf Y)=(\mathbf t,\mathbf y),F_n^1}\big|\F'_{\boldsymbol t}\big]\Big|\F_{\boldsymbol t}\Big]\\
	&=  \prod_{k\in K}e^{\beta H_{[k-\ell_n,k)}(\omega,\pi)-\ell_n\lambda(\beta)}\P\Big((\mathbf T,\mathbf Y)=(\mathbf t,\mathbf y),F_n^1\Big|\F\Big).
\end{align*}
Note that the ``lattice condition'', \cite[first display of Theorem 3]{P74}, is always satisfied for product measures.
\end{proof}
\section{Proof of Theorem~\ref{thm:upper_lower}: Upper bound}\label{sec:upper} 

\subsection{The quadratic variation}\label{sec:corrector}

We compute the quadratic variation of the martingale introduced in \eqref{eq:def_MG}. This calculation does not rely on \eqref{eq:WD}. Recalling the notation from \eqref{eq:def_reverse_f}, we can define the quadratic variation of $(M^f_{n,m})_{m=0,\dots,n}$ by
\begin{equation}\label{eq:comp_f}
\begin{split}
\langle M_{n,\cdot}^f\rangle_m&\coloneqq\oldconstant{c:compensator} n^{-d}\sum_{(t,x)\in[1,m]\times \Z^d}\Big(\sum_{y\in\Z^d} f(y/\sqrt n) W_{t-1}^{0,y}\big[\1_{X_t=x}\big]\Big)^2\\
&=\oldconstant{c:compensator} n^{-d}\sum_{(t,x)\in[1,m]\times \Z^d}\Big(\cev W^{t,x}_1[f(X_0/\sqrt n)]\Big)^2
\end{split}
\end{equation}
where $\newconstant\label{c:compensator}\coloneqq  e^{\lambda(2\beta)-2\lambda(\beta)}-1$. Note that $\langle M_{n,\cdot}^f\rangle_m$ is $\F_{m-1}$-measurable. 

\begin{proposition}\label{prop:corrector}
Let $f\colon\R^d\to\R$ be compactly supported and recall \eqref{eq:def_MG} and \eqref{eq:comp_f}. Then $((M_{n,m}^f)^2-\langle M_{n,\cdot}^f\rangle_m)_{m=0,\dots,n}$ is a martingale.
\end{proposition}

\begin{proof}
Using the reversibility of the simple random walk, we have
\begin{align*}
W_t^{0,y}-W_{t-1}^{0,y}&=\sum_{x\in\Z^d}E^{0,y}\left[\left(e^{\beta H_{[1,t]}(\omega,X)-t\lambda(\beta)}-e^{\beta H_{[1,t-1]}(\omega,X)-(t-1)\lambda(\beta)}\right)\1_{X_t=x}\right]\\
&=\sum_{x\in\Z^d}\big(e^{\beta\omega_{t,x}-\lambda(\beta)}-1\big)E^{0,y}\left[e^{\beta H_{[1,t-1]}(\omega,X)-(t-1)\lambda(\beta)}\1_{X_t=x}\right]\\
&=\sum_{x\in\Z^d}\big(e^{\beta\omega_{t,x}-\lambda(\beta)}-1\big)\cev E^{t,x}\left[e^{\beta H_{[1,t-1]}(\omega,X)-(t-1)\lambda(\beta)}\1_{X_0=y}\right],
\end{align*}
so that
\begin{align*}
\sum_{y\in\Z^d}f(y/\sqrt n)(W_t^{0,y}-W_{t-1}^{0,y})=\sum_{x\in\Z^d}(e^{\beta\omega_{t,x}-\lambda(\beta)}-1)\cev W_1^{t,x}[f(X_0/\sqrt n)].
\end{align*}
Hence
\begin{align*}
\langle M_{n,\cdot}^f\rangle_m&=\sum_{t=1}^m\E\left[(M_{n,t}^f-M_{n,t-1}^f)^2\big|\F_{t-1}\right]\\
&=n^{-d}\sum_{t=1}^m\E\Big[\Big( \sum_{y\in\Z^d}f(y/\sqrt n)(W_t^{0,y}-W_{t-1}^{0,y})\Big)^2\Big|\F_{t-1}\Big]\\
&=\oldconstant{c:compensator}n^{-d}\sum_{(t,x)\in[1,m]\times\Z^d}\Big(\cev W_1^{t,x}[f(X_0/\sqrt n)]\Big)^2,
\end{align*}
where in the last line we used that $\E\left[\big(e^{\beta\omega_{t,x}-\lambda(\beta)}-1\big)\big(e^{\beta\omega_{t,x'}-\lambda(\beta)}-1\big)\right]=\oldconstant{c:compensator}\1_{x=x'}$.
\end{proof}

\subsection{Upper bound on the corrector}\label{sec:upper_corrector}

In this section we prove the first part of Theorem~\ref{thm:upper_lower}(ii), namely \eqref{eq:upper_bound}, according to the following strategy.

\smallskip To obtain an upper bound for the sum in \eqref{eq:comp_f}, we argue that the index set can be changed to ``$(t,x)\in[0,n]\times[-n^{1/2+o(1)},n^{1/2+o(1)}]^d$'' and then drop the term ``$f(X_0/\sqrt n)$'', which leads to \eqref{eq:corrector_approx}. For the remaining sum, we group the summands into levels depending on the value of $\cev W_{1}^{t,x}$. For each level, there is a competition between the number of sites $(t,x)$  attaining this level and the contribution from each site. 

\smallskip For $\p\in(1,2)$, it turns out that the dominant contribution to \eqref{eq:corrector_approx} is from the highest level set, $\cev W_1^{t,x}=n^{(1+d/2)/\p-o(1)}$, which is attained by $n^{o(1)}$ terms. Heuristically, $((W^{t,x}_1)^2)_{t,x}$ behaves like an i.i.d. sequence of random random variables whose tail decays like $u^{-\p/2}$ for $u\to\infty$, so $\p<2$ means that the sum \eqref{eq:corrector_approx} does not satisfy a law of large numbers. Recall also the discussion following Theorem~\ref{thm:main}. Since the scaling limit of such a sum is a stable random variable, we expect that a potential scaling limit analog to Theorem~\ref{thmx:ew}(ii) for $\beta>\beta_{cr}^{L^2}$ should be a suitable stable analog of the Gaussian Free Field.

\begin{proof}[Proof of \eqref{eq:upper_bound}]
Choose $L>0$ such that the support of $f$ is contained in the interior of $[-L,L]^d$ and fix $\delta>0$ to be chosen later. We start by taking care of the contribution to \eqref{eq:comp_f} from $(t,x)$ that are far from the origin, i.e., $|x|_\infty\geq n^{1/2+\delta}$. Note that 
\begin{align*}
\cev W_{1}^{t,x}\big[\1_{X_0\in [-Ln^{1/2},Ln^{1/2}]^d}\big]\isDistr W_{t-1}\big[\1_{X_t\in x+[-Ln^{1/2},Ln^{1/2}]^d}\big].
\end{align*}
There exists $c>0$ such that, for all $n\in\N$, $A>0$ and $(t,x)\in[1,n]\times[-n,n]^d$ with $|x|_\infty>n^{1/2+\delta}$,
\begin{align*}
&\P\left(\big|\cev W_{1}^{t,x}\big[f(X_0/\sqrt n)\big]\big|>\|f\|_\infty A\right)\\
&\leq \P\left(W_{t-1}\big[\1_{X_t\in x+[-Ln^{1/2},Ln^{1/2}]^d}\big]>A\right)\\
&\leq A^{-1}\E\left[W_{t-1}\big[\1_{X_t\in x+[-Ln^{1/2},Ln^{1/2}]^d}\big]\right]\\
&=A^{-1}\PSRW\big(X_t\in x+[-Ln^{1/2},Ln^{1/2}]^d\big)\\
 &\leq A^{-1}\PSRW\big(X_n\cdot e_1\geq n^{1/2+\delta}-Ln^{1/2})\\
 &\leq A^{-1}e^{-cn^{2\delta}},
\end{align*}
where we used the moderate deviation bound for $\PSRW$ from \cite[Theorem~3.7.1]{DZ} in the final inequality. Thus by applying the above bound with $A\coloneqq e^{-cn^{2\delta}/2}/\|f\|_\infty $ together with the union bound we obtain
\begin{equation}\label{eq:bound_super}
\begin{split}
&\P\left(\exists (t,x)\in[1,n]\times[-n,n]^d,|x|_\infty >n^{1/2+\delta}\colon\left(\cev W_{1}^{t,x}\big[f(X_0/\sqrt n)\big]\right)^2\geq e^{-cn^{2\delta}}\right)\\
&\qquad\leq c'n^{1+d}e^{-cn^{2\delta}/2}.
\end{split}
\end{equation}
On the other hand, on the complement of the above event, we have
\begin{align}\label{eq:contr_super}
\sum_{(t,x)\in[1,n]\times[-n,n]^d,|x|_\infty>n^{1/2+\delta}}\left(\cev W_{1}^{t,x}\big[f(X_0/\sqrt n)\big]\right)^2\leq n^{1+d}e^{-cn^{2\delta}}.
\end{align}
It remains to control the contribution from the bulk, i.e., the above sum with $|x|_\infty\leq n^{1/2+\delta}$. To this end, let $K\coloneqq \ceil{\frac{1+d(1/2+\delta)}{\delta \p(1-\delta)}}$ and consider the levels
\begin{align*}
I_0(n)&\coloneqq [0,n^{\delta}),\\
I_{K+1}(n)&\coloneqq \big[n^{(K+1)\delta},\infty\big),\\
I_k(n)&\coloneqq \big[n^{k\delta},n^{(k+1)\delta}\big)\quad\text{ for }k=1,\dots,K,
\end{align*}
together with 
\begin{align*}
\A_k(n)&\coloneqq \left\{(t,x)\in[1,n]\times [-n^{1/2+\delta},n^{1/2+\delta}]^d\colon \cev W_{1}^{t,x}\in I_k\right\}\quad\text{ for }k=0,\dots,K+1,\\
A_k(n)&\coloneqq \left\{|\A_k|\leq n^{1+d(1/2+\delta)-\p k\delta(1-\delta)+\delta}\right\}\quad\text{ for }k=0,\dots,K,\\
A_{K+1}(n)&\coloneqq \{|\A_{K+1}|=0\}.
\end{align*}
For the most part, we will drop the dependence on $n$ to simplify the notation. Note that $\P(A_0)=1$ and, for $k=1,\dots,K$,
\begin{align*}
\P(A_k^c)&\leq n^{-1-d(1/2+\delta)+\p k\delta(1-\delta)-\delta}\E[|\A_k|]\\
&\leq n^{\p k\delta(1-\delta)-\delta}\max_{(t,x)\in [1,n]\times [-n^{1/2+\delta},n^{1/2+\delta}]^d}\P\left(\cev W_{1}^{t,x}\in I_k\right)\\
&\leq n^{\p k\delta(1-\delta)-\delta}\max_{t=1,\dots,n}\P\left(W_{t}\geq n^{k\delta}\right)\\
&\leq n^{-\delta}\sup_k \E\left[W_k^{\p(1-\delta)}\right].
\end{align*}
We have assumed $\p(\beta)>1$, so the supremum in the final line is finite. Similarly,
\begin{align*}
\P(A_{K+1}^c)
\leq n^{-\p\delta(1-\delta)}\sup_k\E\left[W_k^{\p(1-\delta)}\right].
\end{align*}
From the definition of $\p$ and the union bound, we conclude that
\begin{align}\label{eq:bulk_bound}
\lim_{n\to\infty}\P\left(A_0(n)\cap\dots\cap A_{K+1}(n)\right)=1.
\end{align}
On that event, we have
\begin{equation}\label{eq:sliced}
\begin{split}
\sum_{(t,x)} \left(\cev W_{1}^{t,x}\big[f(X_0/\sqrt n)\big]\right)^2
&\leq \|f\|_\infty^2\sum_{(t,x)} \left(\cev W_{1}^{t,x}\right)^2\\
&\leq \|f\|_\infty^2\sum_{k=0}^K|\A_k|n^{2(k+1)\delta}\\
&\leq \|f\|_\infty^2\sum_{k=0}^Kn^{1+d(1/2+\delta)-\p k\delta(1-\delta)+\delta}n^{2(k+1)\delta}\\
&= \|f\|_\infty^2\sum_{k=0}^Kn^{1+d(1/2+\delta)+3\delta} n^{k\delta(-\p (1-\delta)+2)},%\\
\end{split}
\end{equation}
where in the first line we sum over $(t,x)\in[1,n]\times [-n^{1/2+\delta},n^{1/2+\delta}]^d$. If $\p\in(1,2]$, the exponent is increasing in $k$, so each summand in the final line is bounded by
\begin{align*}
n^{1+d(1/2+\delta)+3\delta}n^{\big(\frac{1+d(1/2+\delta)}{\delta \p(1-\delta)}+1\big)\delta(-\p (1-\delta)+2)}\leq n^{\frac{2+d(1+2\delta)}{\p(1-\delta)}+5\delta}.
\end{align*}
The claim follows by choosing $\delta>0$ small enough that $\frac{2+d(1+2\delta)}{\p(1-\delta)}+5\delta<\frac{2+d}{\p}+\eps$ and combining the above bound with \eqref{eq:contr_super}, \eqref{eq:bulk_bound} and  \eqref{eq:bound_super}. On the other hand, if $\p>2$, then we choose $\delta\in(0,1/3)$ small enough that $\p(1-\delta)>2$. Thus the exponent in the final line of \eqref{eq:sliced} is decreasing in $k$ and the sum is bounded by
\begin{align*}
\sum_{(t,x)\in[1,n]\times [-n^{1/2+\delta},n^{1/2+\delta}]^d} \left(\cev W_{1}^{t,x}\big[f(X_0/\sqrt n)\big]\right)^2\leq \|f\|_\infty^2(K+1)n^{1+d/2+3\delta}.\tag*{\qedhere}
\end{align*}
\end{proof}

\subsection{Proof of Theorem~\ref{thm:upper_lower}(ii)}\label{sec:proof_upper}

To complete the proof of Theorem~\ref{thm:upper_lower}(ii), we need to translate the bound on the corrector, \eqref{eq:upper_bound}, into a bound on $\X^f$. In contrast to the lower bound in Section~\ref{sec:proof_lower}, the translation from the quadratic variation to the martingale itself is rather standard for the upper bound. 

\begin{proof}[Proof of \eqref{eq:upper_bound_mg} using \eqref{eq:upper_bound}]
We use the stopping time 
\begin{align*}
\tau_n\coloneqq \inf\left\{t\geq 0:\langle M_{n,\cdot}^f\rangle_{t+1}\geq n^{-d+\frac{2+d}{\p\wedge 2}+\eps}\right\}\in\{0,\dots,n-1\}\cup\{\infty\}.
\end{align*}
Note that $\langle M_{n,\cdot}^f\rangle_{t+1}$ is $\F_t$-measurable. Hence $((M_{n,k\wedge\tau_n}^f)^2-\langle M_{n,\cdot}^f\rangle_{k\wedge \tau_n})_{k=0,\dots,n}$ is a martingale and therefore
\begin{align}\label{eq:L2}
\E\big[(M_{n,n\wedge\tau_n}^f)^2\big]=\E[\langle M_{n,\cdot}^f\rangle_{n\wedge \tau_n}]\leq n^{-d+\frac{2+d}{\p\wedge 2}+\eps}.
\end{align}
Now
\begin{align*}
\P\left(M_{n,n}^f\geq n^{-d/2+\frac{2+d}{2\p\wedge 4}+\eps}\right)&\leq \P(\tau_n\leq n)+\P\left(M_{n,n}^f\geq n^{-d/2+\frac{2+d}{2\p\wedge 4}+\eps},\tau_n=\infty\right)\\
&\leq \P(\tau_n\leq n)+\P\left(M_{n,n\wedge\tau_n}^f\geq n^{-d/2+\frac{2+d}{2\p\wedge 4}+\eps}\right).
\end{align*}
The first term converges to zero by \eqref{eq:upper_bound} whereas for the second term, by \eqref{eq:L2},
\begin{align*}
\P\left(M_{n,n\wedge\tau_n}^f\geq n^{-d/2+\frac{2+d}{2\p\wedge 4}+\eps}\right)\leq \frac{\E[(M_{n,n\wedge\tau_n}^f)^2]}{n^{-d+\frac{2+d}{\p\wedge 2}+2\eps}}\leq n^{-\eps}.\tag*{\qedhere}
\end{align*}
\end{proof}

\section{Equality of $\p$ and $\q$: Proof of Theorem~\ref{thm:pq}}\label{sec:pq}

Throughout this section, the assumptions of Theorem~\ref{thm:pq} are in place.

\smallskip In the following, we show that $\E[W_n^{\p+\eps}]$ diverges exponentially fast for any $\eps>0$. The main idea is to show ``strong localization conditional on $W_n$ attaining a large value'', where \emph{strong localization} refers to a phenomenon in strong disorder, see Theorem~\ref{thmx:phase}(iv). Namely, we show that there exists $c=c(\beta)>0$ such that for all $u$ large enough,
\begin{align}\label{eq:strong_loc_wd}
\P(W_n>u)\approx\P\Big(W_n>u,\max_{x\in\Z^d}\mu_{\omega,n}^\beta(X_n=x)>c\Big).
\end{align}
To explain how that bound is useful, let $\tau(u)\coloneqq \inf\{k:W_k>u\}$ and recall from Theorem~\ref{thmx:moments} that for some $c'$ and all $u>1$, 
\begin{align*}
\P(\tau(u)<\infty)\geq c'u^{-\p}.
\end{align*}
On the other hand, we have, on $\{\tau(u)\leq n,\mu_{\omega,\tau(u)}^\beta(X_{\tau(u)}=x)>c\}$,
\begin{align*}
W_n\geq uc\,W_{n-\tau(u)}\circ\theta_{\tau(u),x}
\end{align*}
and, moreover, $W_{n-\tau(u)}\circ\theta_{\tau(u),x}\isDistr W_{n-\tau(u)}$. Repeating this argument $n$ times and recalling \eqref{eq:strong_loc_wd}, we get $W_{nT}\geq (uc)^n$ with probability $\approx\P(\tau(u)\leq T)^n\approx (c'u^\p)^n$. Since $u$ is arbitrary, we obtain an exponential lower bound by choosing $u$ such that $u^\eps> 1/c^{\p+\eps}c'$.

\smallskip We now proceed to explain the strategy for proving \eqref{eq:strong_loc_wd}, which is strongly influenced by the argument used to prove that \eqref{eq:SD} is equivalent to strong localization, i.e.,
\begin{align*}
	\limsup_{n\to\infty}\max_{x\in\Z^d}\mu_{\omega,n}(X_n=x)>0\qquad\text{ almost surely.}
\end{align*}

Recall the definition \eqref{eq:I_n} of the replice overlap $I_n$. We quickly summarize the argument from \cite{Y10}, which is a variation of the earlier work \cite{CH06} in a related setup. The idea is to study a cleverly defined stochastic process $(X_n)_{n\in\N}$. From its Doob decomposition, one sees that $X_n$ is bounded from below by $\lambda_1 \sum_{k=1}^n I_k-\lambda_2\sum_{k=1}^n\1_{I_k>c}$, where $\lambda_1,\lambda_2>0$ and $c>0$ is explicit.  In strong disorder, it is known that the cumulative replica overlap  $\sum_{k=1}^nI_k$ diverges, see Theorem~\ref{thmx:phase}(iv), so from the fact that $(X_n)_{n\in\N}$ is bounded they can conclude that $\sum_{k=1}^n \1_{I_k>c}$ must diverge as well.

\smallskip Despite working with weak disorder, we can adapt this construction for our purposes (Lemma~\ref{lem:exists}) because the definition of $X_n$ and its Doob decomposition only require the assumption \cite[display $(1.31)$]{Y10}, which corresponds to $\beta>\beta_{cr}^{L^2}$ in our setup. Of course, the final part of their argument does not apply since we know that $\sum_{k=1}^nI_k$ is almost surely bounded in weak disorder. Instead, we show in Lemma~\ref{lem:overlap2} that $\sum_{k=1}^nI_k$ is large conditional on $\{W_n\gg 1\}$, from which we can then conclude that $\sum_{k=1}^n\1_{I_k>c}$ must also be large conditional on $\{W_n\gg 1\}$. This allows us to conclude that $I_k>c$ must hold for some $k\leq n$, and an addition argument (Lemma~\ref{lem:downward}) ensures that we can choose $k\approx n$.

\begin{remark}
Our result \eqref{eq:strong_loc_wd}, formally proved in \eqref{eq:loc_in_wd} below, thus shows that the strong localization phenomenon still plays a role in weak disorder with $\beta>\beta_{cr}^{L^2}$, namely that it occurs on the tail-event $W_n\gg 1$. This is in contrast to the fact that the polymer measure is delocalized  in weak disorder, i.e. $\max_x\mu_{\omega,n}^\beta(X_n=x)$ converges to zero almost surely. Since the latter result is a statement about the typical behavior of $\mu_{\omega,n}^\beta$, there is not contradiction.
\end{remark}

\begin{remark}\label{rm:critical}
It is an intriguing question whether \eqref{eq:strong_loc_wd} is still valid for $\beta\leq\beta_{cr}^{L^2}$. The constant $c(\beta)$ obtained from our proof degenerates as $\beta\downarrow\beta_{cr}^{L^2}$, which suggests (but does not prove) that this phenomenon does not extend further into the weak disorder phase.
\end{remark}

Turning now to the proof, we first show that $\sup_nW_n>u$ implies that the cumulative replica overlap $\sum_{m=\tau(\sqrt u)}^{\tau(u)}I_m$ is at least of order $\log(u)$. To put the result  into perspective, recall that $\P(\tau(\alpha^{2k})<\infty)\approx \alpha^{-2k\p}\gg \alpha^{-4k}$.

\begin{lemma}\label{lem:overlap2}
	Recall \eqref{eq:I_n} and let $\alpha\coloneqq 2e^{\beta K}$. There exists $\newconstant\label{c:1}>0$ such that, for any $k\in\N$,
\begin{align}\label{eq:at_least_log}
\P\Big(\textstyle \sum_{n=\tau(\alpha^k)+1}^{\tau(\alpha^{2k})}I_n\leq \oldconstant{c:1} k,\tau(\alpha^{2k})<\infty\Big)\leq \alpha^{-4k}.
\end{align}
\end{lemma}

 Recall from \eqref{eq:bounded} that $K$ is the upper bound on the environment. The constant \oldconstant{c:1} (as well as \oldconstant{c:3} and \oldconstant{c:4} below) depends on $\alpha$, and hence on $K$ and $\beta$. 
\begin{proof} We define $H_n\coloneqq \alpha^{-\ceil{\log_\alpha W_{n-1}^*}}$, where $W_n^*\coloneqq \sup_{m\leq n}W_m$, and consider the discrete stochastic integral 
\begin{align}\label{eq:stoch_inte}
(H\cdot W)_n\coloneqq \sum_{m=1}^n H_m (W_m-W_{m-1}).
\end{align}
Since $(H_m)_{m\in\N}$ is previsible, $((H\cdot W)_n)_{n\in\N}$ is a martingale with bounded increments, 
\begin{align*}
(H\cdot W)_{n+1}-(H\cdot W)_n=H_{n+1}(W_{n+1}-W_{n})\quad\begin{cases}
\leq e^{\beta K}-1,\\
\geq -\alpha^{-1}. 
\end{cases}
\end{align*}
For the upper bounded, we have used \eqref{eq:bounded} to guarantee that, almost surely,
\begin{align*}
	W_{n+1}-W_n\leq (e^{\beta K}-1)W_n.
\end{align*}
 Next, we bound the quadratic variation as follows
\begin{equation}\label{eq:qv}
\begin{split}
&\left\langle H\cdot W\right\rangle_{n\wedge \tau(\alpha^{2k})}-\left\langle H\cdot W\right\rangle_{\tau(\alpha^{k})\wedge n}\\
&=\sum_{m=(\tau(\alpha^{k})\wedge n)+1}^{\tau(\alpha^{2k})\wedge n}\E\left[\big((H\cdot W)_m-(H\cdot W)_{m-1}\big)^2\Big|\F_{m-1}\right]\\
&=\sum_{m=(\tau(\alpha^{k})\wedge n)+1}^{\tau(\alpha^{2k})\wedge n}\alpha^{-2\ceil{\log_\alpha(W_{m-1}^*)}}\E\left[(W_m-W_{m-1})^2\big|\F_{m-1}\right]\\
&=\sum_{m=(\tau(\alpha^{k})\wedge n)+1}^{\tau(\alpha^{2k})\wedge n}\alpha^{-2\ceil{\log_\alpha(W_{m-1}^*)}}W_{m-1}^2\\
 &\qquad\qquad\times \E\Big[\Big(\sum_x \mu_{m-1}(X_m=x)\big(e^{\beta \omega_{m,x}-\lambda(\beta)}-1\big)\Big)^2\Big|\F_{m-1}\Big]\\
 &=\oldconstant{c:compensator}\sum_{m=(\tau(\alpha^{k})\wedge n)+1}^{\tau(\alpha^{2k})\wedge n}\alpha^{-2\ceil{\log_\alpha(W_{m-1}^*)}}W_{m-1}^2 I_m\\
 &\leq c\sum_{m=(\tau(\alpha^{k})\wedge n)+1}^{\tau(\alpha^{2k})\wedge n} I_m,
\end{split}
\end{equation}
where we recall that $\oldconstant{c:compensator}=e^{\lambda(2\beta)-2\lambda(\beta)}-1$. Moreover, on $\{\tau(\alpha^{2k})\leq n\}$,
\begin{equation}\label{eq:mg}
\begin{split}
(H\cdot W)_{\tau(\alpha^{2k})\wedge n}-(H\cdot W)_{\tau(\alpha^{k})\wedge n}&=\sum_{i=k+1}^{2k}\big((H\cdot W)_{\tau(\alpha^{i})}-(H\cdot W)_{\tau(\alpha^{i-1})}\big)\\
&\geq \sum_{i=k+1}^{2k} \alpha^{-i}\left(W_{\tau(\alpha^i)}-W_{\tau(\alpha^{i-1})}\right)\\
&\geq c' k.
\end{split}
\end{equation}
In the last inequality, we have used \eqref{eq:bounded} and the definition of $\alpha$ to ensure that, almost surely,  $W_{\tau(\alpha^{i-1})}\leq e^{\beta K}\alpha^{i-1}={\alpha^i}/2$. Combining \eqref{eq:mg} and \eqref{eq:qv}, we thus have, for any $\oldconstant{c:1}>0$,
\begin{equation}\label{eq:obs}
\begin{split}
&\left\{\sum_{m=(\tau(\alpha^k)\wedge n)+1}^{\tau(\alpha^{2k})\wedge n}I_m\leq \oldconstant{c:1} k,\tau(\alpha^{2k})\leq n\right\}\\
& \subseteq \left\{\begin{aligned}&(H\cdot W)_{\tau(\alpha^{2k})\wedge n}-(H\cdot W)_{\tau(\alpha^{k})\wedge n}\geq c^{\prime}k,\\& \left\langle H\cdot W\right\rangle_{\tau(\alpha^{2k})\wedge n}-\left\langle H\cdot W\right\rangle_{\tau(\alpha^{k})\wedge n}\leq c\oldconstant{c:1} k\end{aligned}\right\}.
\end{split}
\end{equation}
Let $\lambda>0$. By \cite[Chapter VII\,\,\S 3 Lemma 1]{S19}, there exists $\psi(\lambda)>0$ such that the discrete stochastic exponential $(E_n^\lambda)_{n\in\N}$ is a non-negative supermartingale, where
\begin{align*}
E_n^\lambda\coloneqq e^{\lambda ((H\cdot W)_{\tau(\alpha^{2k})\wedge n}-(H\cdot W)_{\tau(\alpha^k)\wedge n})-\psi(\lambda)(\left\langle H\cdot W\right\rangle_{\tau(\alpha^{2k})\wedge n}-\left\langle H\cdot W\right\rangle_{\tau(\alpha^k)\wedge n})}.
\end{align*} 
Hence, using  the Markov inequality and $\E[E^\lambda_n]\leq \E[E^\lambda_0]=1$,
\begin{align*}
&\P\left(\begin{aligned}&(H\cdot W)_{n\wedge\tau(\alpha^{2k})}-(H\cdot W)_{n\wedge \tau(\alpha^{k})}\geq c^{\prime}k,
\\ &\left\langle H\cdot W\right\rangle_{n\wedge\tau(\alpha^{2k})}-\left\langle H\cdot W\right\rangle_{n\wedge \tau(\alpha^{k})}\leq c\oldconstant{c:1} k\end{aligned}\right)\\
     &=\P\left(\begin{aligned}&E_n^\lambda\geq e^{\lambda c'k-\psi(\lambda)(\left\langle H\cdot W\right\rangle_{n\wedge \tau(\alpha^{2k})}-\left\langle H\cdot W\right\rangle_{\tau(\alpha^{k})\wedge n})},\\&\left\langle H\cdot W\right\rangle_{n\wedge\tau(\alpha^{2k})}-\left\langle H\cdot W\right\rangle_{n\wedge \tau(\alpha^{k})}\leq c\oldconstant{c:1} k\end{aligned}\right)\\
     &\leq\P\left(E_n^\lambda\geq e^{k(\lambda c'-\psi(\lambda)c\oldconstant{c:1})}\right)\\
     &\leq e^{-k(\lambda c'-\psi(\lambda)c\oldconstant{c:1})}.
\end{align*}
We thus choose $\lambda\coloneqq (4\log(\alpha)+1)/c'$ and $\oldconstant{c:1}\coloneqq \frac{1}{c\psi(\lambda)}$. The claim follows from \eqref{eq:obs} after taking the limit $n\to\infty$.
\end{proof}

Next, we  show  that strong localization occurs in $[\tau(\sqrt u),\tau(u)]$ conditional on $\sup_nW_n>u$.
\begin{lemma}\label{lem:exists}
There exist $\newconstant\label{c:3},\newconstant\label{c:4}>0$ such that, for all $k\geq \oldconstant{c:4}$,
\begin{align*}
	\P\left(\max_{n=\tau(\alpha^k),\tau(\alpha^k)+1,\dots,\tau(\alpha^{2k})-1}\max_{x\in\Z^d}\mu_{\omega,n}(X_n=x) < \oldconstant{c:3},\tau(\alpha^{2k})<\infty\right)\leq 2\alpha^{-4k}.
\end{align*}
\end{lemma}

\begin{proof}
We adapt the proof from \cite[Section 3]{Y10}, which is in turn an adaption of the proof from \cite{CH06} to the discrete-time setting. There, they consider a slightly different notion of replica overlap, namely
\begin{align*}
\mathcal R_n\coloneqq\sum_x\mu_{\omega,n}(X_n=x)^2.
\end{align*}
In fact, this notion is practically equivalent to $I_n$, since, almost surely,
\begin{align}\label{eq:RI}
\frac{1}{4d^2}I_m\leq \mathcal R_{m-1}\leq 4d^2I_m.
\end{align}
We also note that $\mathcal R_n\leq \max_{x\in\Z^d}\mu_{\omega,n}(X_n=x)$. Let $g\in[0,\infty)^{\Z^d}$ be the quantity from \cite[display $(3.3)$]{Y10}. We consider the process
\begin{align*}
X_n\coloneqq 
\begin{cases}
\mu_{\omega,n}\cdot(\mu_{\omega,n}*g)-\mu_{\omega,\tau(\alpha^k)}\cdot(\mu_{\omega,\tau(\alpha^k)}*g)&\text{ if }n\geq\tau(\alpha^k),\\0&\text{ else. }
\end{cases}
\end{align*}
Here, ``$\cdot$'' refers to the scalar product in $\R^{\Z^d}$ and ``$*$'' to the convolution operator on $\R^{\Z^d}$. We consider the Doob decomposition $X_n=A_n+M_n,$ with $(A_n)_{n\in\N}$ previsible and $(M_n)_{n\in\N}$ a martingale. The next lemma collects results from \cite{Y10} and will be proved further below. 

\begin{lemma}\label{lem:byY10}
There exist $c,c',c''$ and $c'''>0$ such that, almost surely for all $n\in\N$,
\begin{align}
	|X_n|&\leq c.\label{eq:X_n}\\
	A_n& \geq c'\sum_{m=(\tau(\alpha^k)\wedge n)+1}^n \mathcal R_{m-1}-c''\sum_{m=(\tau(\alpha^k)\wedge n)+1}^n \mathcal R_{m-1}^{3/2}.\label{eq:A_n}\\
\left\langle M\right\rangle_n&\leq c'''\sum_{m=(\tau(\alpha^k)\wedge n)+1}^n\mathcal R_{m-1}^2.\label{eq:M_n}
\end{align}
\end{lemma}

To conclude, we introduce the following events
\begin{align*}
	B_1&\coloneqq \Big\{A_n> \frac{c'}{2}\sum_{m=\tau(\alpha^k)+1}^{\tau(\alpha^{2k})}\mathcal R_{m-1}\Big\},\\
	B_2&\coloneqq \Big\{\sum_{m=\tau(\alpha^k)+1}^{\tau(\alpha^{2k})}\mathcal R_{m-1}>\frac{\oldconstant{c:1}}{4d^2}k\Big\},\\
	B_3&\coloneqq \Big\{-M_{\tau(\alpha^{2k})}<\frac{c'}{4}\sum_{m=\tau(\alpha^k)+1}^{\tau(\alpha^{2k})}\mathcal R_{m-1}\Big\},\\
	R_\delta&\coloneqq \Big\{\max_{m=\tau(\alpha^k)+1,\dots,\tau(\alpha^{2k})}\mathcal R_{m-1}\leq\delta\Big\}.
\end{align*}
We claim that, for $k$ large enough, 
\begin{align}
	B_1\cap B_2\cap B_3\cap\big\{\tau(\alpha^{2k})<\infty\big\}&=\varnothing,\label{eq:claim0}\\
B_1^c\cap R_{(c'/4c'')^2} \cap\big\{\tau(\alpha^{2k})<\infty)\big\}&=\varnothing,\label{eq:claim1}\\
	\P\big(B_2^c,\tau(\alpha^{2k})<\infty\big)&\leq \alpha^{-4k},\label{eq:claim2}\\
	\exists \eps\in(0,1)\text{ s.t. }\P(B_3^c,B_2,R_{\eps/c'''},\tau(\alpha^{2k})<\infty)&\leq \alpha^{-4k}.	\label{eq:claim3}
\end{align}
If we set $\oldconstant{c:3}\coloneqq (\frac{c'}{4c''})^2\wedge \frac{\eps}{c'''}$, then the conclusion follows from these claims, since 
\begin{align*}
	R_{(\frac{c'}{4c''})^2\wedge \frac{\eps}{c'''}}\subseteq \left(B_1\cap B_2\cap B_3\right)\cup \big(B_1^c\cap R_{(\frac{c'}{4c''})^2}\big)\cup B_2^c\cup \big(B_3^c\cap B_2\cap R_{\frac{\eps}{c'''}}\big).
\end{align*}
For \eqref{eq:claim0}, we note that, on $B_1\cap  B_3\cap\{\tau(\alpha^{2k})<\infty\}$, 
\begin{align*}
\frac{c'}{2}\sum_{m=\tau(\alpha^k)+1}^{\tau(\alpha^{2k})}\mathcal R_{m-1}< A_n=X_n-M_n\leq c+\frac{c'}{4}\sum_{m=\tau(\alpha^k)+1}^{\tau(\alpha^{2k})}\mathcal R_{m-1}.
\end{align*}
In particular, \eqref{eq:claim0} holds for $k> \frac{c16d^2}{c'\oldconstant{c:1}}=:\oldconstant{c:4}$. 
For \eqref{eq:claim1}, we use the fact that $x^{3/2}\leq \1_{x> \eps}+\sqrt\eps x$ for $x,\eps\in[0,1]$ to get
\begin{align*}
&B_1^c\cap\big\{\tau(\alpha^{2k})<\infty\big\}\\
&\subseteq \left\{\frac{c'}{2c''}\sum_{m=\tau(\alpha^k)+1}^{\tau(\alpha^{2k})}\mathcal R_{m-1}\leq \sum_{m=\tau(\alpha^k)+1}^{\tau(\alpha^{2k})}\mathcal R_{m-1}^{3/2},\tau(\alpha^{2k})<\infty\right\}\\
&\subseteq \left\{\frac{c'}{4c''}\sum_{m=\tau(\alpha^k)+1}^{\tau(\alpha^{2k})}\mathcal R_{m-1}\leq \sum_{m=\tau(\alpha^k)+1}^{\tau(\alpha^{2k})}\1_{\mathcal R_{m-1}> ({c'}/{4c''})^2},\tau(\alpha^{2k})<\infty\right\}.
\end{align*}
Since $\mathcal R_{m-1}$ is almost surely positive, \eqref{eq:claim1} follows. For \eqref{eq:claim2}, \eqref{eq:RI} and Lemma~\ref{lem:overlap2} show
\begin{align*}
	\P(B_2^c,\tau(\alpha^{2k})<\infty)\leq \P\Big(\sum_{m=\tau(\alpha^k)+1}^{\tau(\alpha^{2k})}I_m\leq \oldconstant{c:1}k,\tau(\alpha^{2k})<\infty\Big)\leq \alpha^{-4k}.
\end{align*}
Finally, to prove \eqref{eq:claim3}, we use \eqref{eq:M_n} to obtain that, on $\{\tau(\alpha^{2k})<\infty\}\cap R_{\eps/c'''}$,
\begin{align}\label{eq:qv_small}
\left\langle M\right\rangle_{\tau(\alpha^{2k})}\leq\eps\sum_{m=\tau(\alpha^k)+1}^{\tau(\alpha^{2k})}\mathcal R_{m-1}.
\end{align}
For $\lambda>0$, let
\begin{align*}
E^\lambda_n\coloneqq e^{-\lambda M_{\tau(\alpha^{2k})\wedge n}-\psi(\lambda)\left\langle M\right\rangle_{\tau(\alpha^{2k})\wedge n}}
\end{align*}
denote the discrete exponential supermartingale defined in  \cite[Chapter VII\,\,\S 3, Lemma 1]{S19}. We note that the result applies because $(X_n)_{n\in\N}$, and hence $(M_n)_{n\in\N}$, has bounded increments. By Ville's inequality, we have
\begin{align*}
&\P\left(B_3^c,B_2,R_{\frac{\eps}{c'''}},\tau(\alpha^{2k})<\infty\right)\\
   &\leq \P\left(\sup_mE_{m}^\lambda\geq e^{\frac{c'\lambda}4 \sum_{m=\tau(\alpha^k)+1}^{\tau(\alpha^{2k})}\mathcal R_{m-1}-\psi(\lambda)\left\langle M\right\rangle_{\tau(\alpha^{2k})}},B_2,R_{\frac{\eps}{c'''}},\tau(\alpha^{2k})<\infty\right)\\
   &\leq \P\left(\sup_mE_{m}^\lambda\geq e^{(\lambda c/4-\eps\psi(\lambda)) \sum_{m=\tau(\alpha^k)+1}^{\tau(\alpha^{2k})}\mathcal R_{m-1}}, B_2\right)\\
   &\leq \P\left(\sup_mE_{m}^\lambda\geq e^{k\oldconstant{c:1} (\lambda c/4-\eps\psi(\lambda))/4d^2}\right)\\
   &\leq e^{-k\oldconstant{c:1} (\lambda c/4-\eps\psi(\lambda))/4d^2}.
\end{align*}
We have used \eqref{eq:qv_small} in the second inequality. Now \eqref{eq:claim3} follows with $\lambda\coloneqq 4(\frac{16d^2\log(\alpha)}{\oldconstant{c:1}}+1)/c$ and $\eps\coloneqq 1/\psi(\lambda)\wedge \frac{1}{2}$.
\end{proof}
It remains to justify the claims taken from \cite{Y10} in Lemma~\ref{lem:byY10}.
\begin{proof}[Proof of Lemma~\ref{lem:byY10}]
In Section 3 of \cite{Y10}, they consider the process 
\begin{align*}
\widetilde X_n\coloneqq\mu_{\omega,n}\cdot(\mu_{\omega,n}*g).
\end{align*}
Recalling \eqref{eq:stoch_inte}, we observe that $X$ can be written as a stochastic integral of $\widetilde X$, 
\begin{align*}
	X_n=\widetilde{X}_n-\widetilde{X}_{n\wedge \tau(\alpha^k)}=\sum_{m=1}^n H_m(\widetilde{X}_m-\widetilde{X}_{m-1})= (H\cdot \widetilde X)_n,
\end{align*}
where $H_n\coloneqq \1_{\tau(\alpha^k)<n}$. Let $\widetilde{A}_n+\widetilde{M}_n$ be the Doob decomposition of $\widetilde{X}_n$. Since $H_n$ is previsible, it is easy to check that
 \begin{align}
	 A_n&=\sum_{m=1}^n H_m (\widetilde{A}_m-\widetilde{A}_{m-1})=\sum_{m=(\tau(\alpha^k)\wedge n)+1}^n(\widetilde{A}_{m}-\widetilde{A}_{m-1})\label{eq:related},\\
	 M_n&=\sum_{m=1}^n H_m (\widetilde{M}_m-\widetilde{M}_{m-1})=\sum_{m=(\tau(\alpha^k)\wedge n)+1}^n(\widetilde{M}_m-\widetilde{M}_{m-1}).\label{eq:related2}
\end{align}
By \cite[display $(3.4)$]{Y10}, it holds that $|\widetilde X_n|\leq |g|_1<\infty$, so \eqref{eq:X_n} follows. Next, we recall that by \cite[last display on p.17]{Y10} and \cite[display (3.9)]{Y10} there exist $c',c''>0$ such that, almost surely for all $m\in\N$,
\begin{align*}
	\widetilde A_m-\widetilde A_{m-1}\geq c' \mathcal R_{m-1}-c'' \mathcal R_{m-1}^{3/2}.
\end{align*}
Thus \eqref{eq:A_n} follows from \eqref{eq:related}. Moreover, by \cite[Section 3.3, displays \textbf{2)} and \textbf{3)}]{Y10}, almost surely for all $m\in\N$, 
\begin{align*}
	\langle M\rangle_m-\langle M\rangle_{m-1}&\leq 8|g|_1^2\Big(\E[\mathcal R_{m}^2|\F_{m-1}]+\mathcal R_{m-1}^2\Big)
\end{align*}
Using \eqref{eq:related2}, it is thus enough to show that, almost surely for all $m\in\N$, 
\begin{align*}
	\E[\mathcal R_m^2|\F_{m-1}]\leq (2d)^{7/2}e^{\frac{1}{2}\lambda(8\beta)+\frac{1}{2}\lambda(-8\beta)}\mathcal R_{m-1}^2
\end{align*}
Indeed, using the inequality $(\sum_{i\in I}a_i)^2\leq|I|\sum_{i\in I}a_i^2$, we obtain
\begin{equation}\label{eq:above}
	\begin{split}\sum_{y\in\Z^d}W_m[\1_{X_m=y}]^2&=\sum_{y\in\Z^d}\Big(\sum_{|x-y|_1=1}W_{m}[\1_{X_{m-1}=x,X_m=y}]\Big)^2\\
&\leq 
\sum_{x\in\Z^d}W_{m-1}[\1_{X_{m-1}=x}]^2 A(x),
\end{split}
\end{equation}
where $A(x)\coloneqq{2d}\sum_{|y|_1=1} e^{2\beta\omega_{m,x+y}-2\lambda(\beta)}$. We introduce the probability measure 
\begin{align*}
\nu(x)\coloneqq\frac{W_{m-1}[\1_{X_{m-1}=x}]^2}{\sum_{y\in\Z^d}W_{m-1}[\1_{X_{m-1}=y}]^2}
\end{align*}
and observe that, by \eqref{eq:above},
\begin{align*}
\Big(\sum_{y\in\Z^d}W_m[\1_{X_m=y}]^2\Big)^2\leq \Big(\sum_{y\in\Z^d}W_{m-1}[\1_{X_{m-1}=y}]^2\Big)^2\Big(\sum_{x\in\Z^d}\nu(x) A(x)\Big)^2
\end{align*}
and therefore
\begin{align*}
\mathcal R_m^2\leq \mathcal R_{m-1}^2\frac{\big(\sum_{x\in\Z^d}\nu(x) A(x)\big)^2}{\big(\sum_{x\in\Z^d}\mu_{\omega,m-1}(X_m=x)e^{\beta\omega_{m,x}-\lambda(\beta)}\big)^4}.
\end{align*}
Both $\nu$ and $\mu_{\omega,m-1}$ are $\F_{m-1}$-measurable. Thus, using Cauchy-Schwarz and Jensen's inequality, 
\begin{align*}
&\E\Big[\frac{\big(\sum_x\nu(x) A(x)\big)^2}{\big(\sum_x\mu_{\omega,m-1}(X_m=x)e^{\beta\omega_{m,x}-\lambda(\beta)}\big)^4}\Big|\F_{m-1}\Big]\\
&\leq \E\Big[\Big(\sum_x\nu(x) A(x)\Big)^4\Big|\F_{m-1}\Big]^{1/2}\E\Big[\Big(\sum_x\mu_{\omega,m-1}(X_m=x)e^{\beta\omega_{m,x}-\lambda(\beta)}\Big)^{-8}\Big|\F_{m-1}\Big]^{1/2}\\
&\leq \E[A(0)^4]^{1/2}\E[e^{-8\beta\omega_{m,0}+8\lambda(\beta)}]^{1/2}\\
&\leq (2d)^{7/2}e^{\frac 12\lambda(8\beta)+\frac 12\lambda(-8\beta)}.\qedhere
\end{align*}

\end{proof}

Next, we record a simple  upper tail bound:

\begin{lemma}\label{lem:almost_sure_upper}
Almost surely for any $t\in\N$, $A>0$ and $\lambda>1$,
\begin{align*}
\P\Big(\sup_n\sum_x\mu_{\omega,t}(X_t=x)W_n\circ\theta_{t,x}\geq A\Big|\F_t\Big)\leq A^{-\lambda}\E\big[\sup_nW_n^\lambda\big].
\end{align*}
\end{lemma}
\begin{proof}Writing $\mu_{\omega,t}(x)$ instead of $\mu_{\omega,t}(X_t=x)$, we have
\begin{align*}
\P\Big(\sup_n\sum_x\mu_{\omega,t}(x)W_n\circ\theta_{t,x}\geq A\Big|\F_t\Big)&\leq A^{-\lambda}\E\left[\left.\Big(\sup_n\sum_x\mu_{\omega,t}(x)W_n\circ\theta_{t,x}\Big)^\lambda \right|\F_t\right]\\
&\leq A^{-\lambda}\E\left[\left.\Big(\sum_x\mu_{\omega,t}(x)\sup_n W_n\circ\theta_{t,x}\Big)^\lambda \right|\F_t\right]\\
&\leq A^{-\lambda}\E\left[\left.\sum_x\mu_{\omega,t}(x)\Big(\sup_nW_n\circ\theta_{t,x}\Big)^\lambda \right|\F_t\right]\\
&= A^{-\lambda}\E\big[\sup_nW_n^\lambda\big],
\end{align*}
where the third inequality is Jensen's inequality. 
\end{proof}
Lemma~\ref{lem:exists} guarantees the existence of a localization time $m\in[\tau(\sqrt u),\tau(u)]$, where $\max_x\mu_{\omega,m}(X_m=x)\geq \oldconstant{c:3}$, but it may be the case that $W_m\ll \sqrt u$. In the next lemma, we exclude that possibility by showing that $(W_n)_n$ does not ``backtrack'' too much after reaching a certain level.

\begin{lemma}\label{lem:downward}
For every $\eta\in(0,1)$, there exist $\delta=\delta(\eta)>0$ and $\newconstant\label{c:2}=\oldconstant{c:2}(\eta)>0$ such that, for all $k\in\N$,
\begin{equation}\label{eq:backtrack}\begin{split}
&\P\left(\exists i\in\{k,\dots,2k-1\}:\min_{m=\tau(\alpha^{i}),\tau(\alpha^i)+1,\dots,\tau(\alpha^{2k})}W_m\leq \alpha^{i-k\eta},\tau(\alpha^{2k})<\infty\right)\\
&\leq \oldconstant{c:2}k\alpha^{-2k(\p+\delta)}.
\end{split}\end{equation}
\end{lemma}

\begin{proof}[Proof of Lemma~\ref{lem:downward}]
For $i\in\{k,\dots,2k-1\}$, let 
\begin{align*}
\sigma_i\coloneqq \inf\{m\in \{\tau(\alpha^{i}),\tau(\alpha^i)+1,\dots,\tau(\alpha^{2k})\}\colon W_m\leq \alpha^{i-k\eta}\}
\end{align*}
and note that \eqref{eq:backtrack} is bounded by $k\max_{i=k,\dots,2k-1}\P(\sigma_i<\infty)$. Using Lemma~\ref{lem:almost_sure_upper}, we get
\begin{align*}
\P\left(\sigma_i<\infty\right)
=&\sum_{s\leq t}\E\left[\1_{\tau(\alpha^{i})=s}\E\left[\1_{\sigma_i=t}\P\big(\tau(\alpha^{2k})<\infty\big|\F_t\big)\Big|\F_s\right]\right]\\
\leq  &\sum_{s\leq t}\E\left[\1_{\tau(\alpha^{i})=s}\E\Big[\1_{\sigma_i=t}\P\Big(\sup_n\sum_x\mu_{\omega,t}(x)W_n\circ\theta_{t,x}\geq \alpha^{2k-i+k\eta}\Big|\F_t\Big)\Big|\F_s\Big]\right]\\
\leq &\alpha^{-\lambda ((2k-i)+k\eta)}\E\big[\sup_nW_n^\lambda\big]\sum_{s\leq t}\E\left[\1_{\tau(\alpha^{i})=s}\E\left[\1_{\sigma_i=t}\Big|\F_s\right]\right]\\
\leq &\alpha^{-\lambda ((2k-i)+k\eta)}\E\big[\sup_nW_n^\lambda\big]\P\big(\sup_nW_n\geq \alpha^i\big)\\
\leq &\alpha^{-\lambda (2+\eta)k}\E\big[\sup_nW_n^\lambda\big]^2.
\end{align*}
We set $\lambda=\p(1-\eta/2)$, so that \eqref{eq:backtrack} holds with $\delta\coloneqq \eta^2\p/2$ and $\oldconstant{c:2}\coloneqq \E[\sup_nW_n^\lambda]^2$.
\end{proof}

Finally, we prove $\p=\q$ by following the idea outlined at the beginning of this section.

\begin{proof}[Proof of Theorem~\ref{thm:pq}]
Let $\eta\coloneqq \frac{\eps}{2(\p+\eps+1)}$. For $i\in\{k,\dots,2k-1\}$, we define
\begin{align*}
A(i)\coloneqq \left\{\exists m\in\{\tau(\alpha^i),\tau(\alpha^i)+1,\dots,\tau(\alpha^{i+1})-1\}:\max_x\mu_{\omega,m}(X_m=x)\geq\oldconstant{c:3},W_m\geq \alpha^{i-k\eta}\right\}.
\end{align*}
By Lemmas~\ref{lem:downward} and~\ref{lem:exists}, for all $k\geq \oldconstant{c:4}$,
\begin{align*}
\P\left(\left(\textstyle \bigcup_{i=k}^{2k-1}A(i)\right)^c,\tau(\alpha^{2k})<\infty\right)\leq 2\alpha^{-4k}+\oldconstant{c:2}(\eta)k\alpha^{-2k(\p+\delta(\eta))}.
\end{align*}
Comparing with \eqref{eq:tail}, we see that there exists $k_0\geq \oldconstant{c:4}$ such that, for all $k\geq k_0$,
\begin{align*}
\P\left(\textstyle \bigcup_{i=k}^{2k-1}A(i),\tau(\alpha^{2k})<\infty\right)\geq \oldconstant{c:tail}\alpha^{-2k\p}/2.
\end{align*}
Thus there exists $l=l(k)\in\{k,\dots,2k-1\}$ such that 
\begin{align}\label{eq:better_lower}
\P\left(A(l),\tau(\alpha^{2k})<\infty\right)\geq \oldconstant{c:tail}\alpha^{-2k\p}/2k.
\end{align}
We define a stopping time $\sigma$ by
\begin{align*}
\sigma\coloneqq \inf\left\{n\in\{\tau(\alpha^l),\tau(\alpha^l)+1,\dots,\tau(\alpha^{l+1})-1\}:\max_x\mu_{\omega,n}(X_n=x)\geq\oldconstant{c:3},W_n\geq \alpha^{l-k\eta}\right\}.
\end{align*}
Using \eqref{eq:better_lower} and Lemma~\ref{lem:almost_sure_upper}, we have, for any $\lambda>1$,
\begin{align*}
\oldconstant{c:tail}\alpha^{-2k\p}/2k&\leq \sum_{m}\E\left[\1_{\sigma=m}\P\left(\tau(\alpha^{2k})<\infty\Big|\F_m\right)\right]\\
&\leq \sum_m\E\Big[\1_{\sigma=m}\P\Big(\sup_n\sum_x\mu_{\omega,m}(x)W_n\circ\theta_{m,x}\geq \alpha^{2k-l}\Big|\F_m\Big)\Big]\\
&\leq \alpha^{-(2k-l)\lambda}\E\big[\sup_nW_n^\lambda\big]\P\left(\sigma<\infty\right).
\end{align*}
We have used that, by construction, $\sigma<\tau(\alpha^{l+1})$. Setting $\lambda\coloneqq \p-\eta$, we get
\begin{align}\label{eq:loc_in_wd}
	\P(\sigma<\infty)\geq c\alpha^{-\p l-(2k-l)\eta}/k,
\end{align}
where $c$ is independent of $k$. With this estimate in hand, we now consider the following renewal construction: On $\sigma<\infty$, let $Z\in\Z^d$ be such that $\mu_{\omega,\sigma}(X_\sigma=Z)\geq \oldconstant{c:3}$. We set $\sigma_0=0$, $Z_0\coloneqq 0$ and then, recursively given $\sigma_i<\infty$ and $Z_i$,
\begin{align*}
\sigma_{i+1}&\coloneqq \sigma\circ\theta_{\sigma_i,Z_i},\\
Z_{i+1}&\coloneqq Z\circ\theta_{\sigma_i,Z_i}\qquad(\text{if }\sigma_{i+1}<\infty).
\end{align*}
Let $T=T(k)\in\N$ be such that
\begin{align}\label{eq:s2}
\P(\sigma\leq T)\geq \P(\sigma<\infty)/2
\end{align}
and $R\coloneqq \inf\{i:\sigma_i-\sigma_{i-1}>T\}$. Clearly, $R$ is geometrically distributed with success parameter $\P(\sigma> T)$. On the event
\begin{align*}
  	\left\{R>\left\lfloor\frac{n}{T}\right\rfloor\right\}\cap\Big\{\inf_{m\in\N}W_m\circ\theta_{\sigma_{\lfloor\frac{n}{T}\rfloor},Z_{\lfloor\frac{n}{T}\rfloor }}>a\Big\}
\end{align*}
we have
\begin{equation}\label{eq:llll}
	\begin{split}
 	W_n&\geq W_n[\1_{X_{\sigma_i}=Z_i\text{ for all }i=1,\dots,\lfloor\frac{n}{T}\rfloor}]\\
 	   &=W_{n-\sigma_{\lfloor\frac{n}{T}\rfloor}}\circ\theta_{\sigma_{\lfloor\frac{n}{T}\rfloor},Z_{\lfloor\frac{n}{T}\rfloor}}\prod_{i=1}^{\lfloor\frac{n}{T}\rfloor}W_{\Delta\sigma_i}[\1_{X_{\Delta\sigma_i}=\Delta Z_i}]\circ\theta_{\sigma_{i-1},Z_{i-1}}\\
 	   &=W_{n-\sigma_{\lfloor\frac{n}{T}\rfloor}}\circ\theta_{\sigma_{\lfloor\frac{n}{T}\rfloor},Z_{\lfloor\frac{n}{T}\rfloor}}\prod_{i=1}^{\lfloor\frac{n}{T}\rfloor}W_{\Delta\sigma_i}\circ\theta_{\sigma_{i-1},Z_{i-1}}\mu_{\theta_{\sigma_{i-1},Z_{i-1}}\omega,\Delta\sigma_i}(X_{\Delta\sigma_i}=\Delta Z_i)\\
 	   &\geq a(\alpha^{l-k\eta}\oldconstant{c:3})^{\lfloor\frac{n}{T}\rfloor},
\end{split}
\end{equation}
where $\Delta\sigma_i=\sigma_i-\sigma_{i-1}$ and $\Delta Z_i=Z_i-Z_{i-1}$. Choosing $a>0$ small enough that $\P(\inf_kW_k>a)\geq \frac 12$, we have
\begin{align*}
\E\left[W_{n}^{\p+\eps}\right]&\geq \P\Big(R>\left\lfloor\frac{n}{T}\right\rfloor,\inf_{m\in\N}W_m\circ\theta_{\sigma_{\left\lfloor\frac{n}{T}\right\rfloor},Z_{\left\lfloor\frac{n}{T}\right\rfloor}}>a\Big)\Big(\oldconstant{c:3} \alpha^{l-k\eta}\Big)^{\left\lfloor\frac{n}{T}\right\rfloor\p+\eps)}a^{\p+\eps}\\
			      &= \P\big(\inf_mW_{m}>a\big)\P(\sigma\leq T)^{\lfloor\frac{n}{T}\rfloor}\big(\oldconstant{c:3}\alpha^{l-k\eta}\big)^{\lfloor\frac{n}{T}\rfloor}a^{\p+\eps}\\
			      &\geq\frac{1}{2} \left(\frac{c\oldconstant{c:3}}{2k}\alpha^{(l-k\eta)(\p+\eps)-\p l-(2k-l)\eta}\right)^{\lfloor\frac{n}{T}\rfloor} a^{\p+\eps}\\
			      &= \frac{1}{2} \left(\frac{c\oldconstant{c:3}}{2k}\alpha^{\eps l-k\eta(\p+\eps)-(2k-l)\eta}\right)^{\lfloor\frac{n}{T}\rfloor} a^{\p+\eps}\\
			      &\geq  \frac{1}{2} \left(\frac{c\oldconstant{c:3}}{2k}\alpha^{k(\eps -\eta(\p+\eps+1))}\right)^{\lfloor\frac{n}{T}\rfloor} a^{\p+\eps}\\
			      &=  \frac{1}{2} \left(\frac{c\oldconstant{c:3}}{2k}\alpha^{k \eps/2}\right)^{\lfloor\frac{n}{T}\rfloor} a^{\p+\eps}.
\end{align*}
We have used \eqref{eq:llll} in the first line, \eqref{eq:loc_in_wd} and \eqref{eq:s2} in the third line, the inequality $l\geq k\geq (2k-l)$ for the fifth line and the definition of $\eta$ in the last line. Finally, we can choose $k$ large enough that the quantity in brackets in the last line is at least $2$, which shows that \eqref{eq:claim} holds with $\oldconstant{c:5}\coloneqq 2^{1/T(k)}>1$ and $\oldconstant{c:useless}\coloneqq \frac{1}{2}(\frac{c\oldconstant{c:3}}{2k}\alpha^{k\eps/2})^{-1}a^{\p+\eps}$.
\end{proof}
 
\begin{appendix}
	\section*{References for Theorems~\ref{thmx:moments} and~\ref{thmx:ew}}\label{appn}
First, we provide the reference for Theorem~\ref{thmx:moments}.

\begin{proof}[Proof of Theorem~\ref{thmx:moments}]
The first claim is proved in \cite[Theorem~1.1$\text{(ii)}$]{J21_1}, where it is also shown that $\sup_n\E[W_n^{\p}]=\infty$. On the other hand, by \cite[display~$\text{(20)}$]{J21_1}, for any $t>1$,
\begin{align}\label{eq:from_proof}
\E\big[(W_n^\beta)^\p\big]\leq t^\p+(te^{\beta K})^\p\P\big({\textstyle\sup_{k=1,\dots,n}}W_k^\beta >t\big)\E\big[(W_n^\beta)^\p\big].
\end{align}
If \eqref{eq:tail} fails for some $t>1$, then \eqref{eq:from_proof} can be rearranged to 
\begin{align*}
\E\big[(W_n^\beta)^\p\big]\leq 2t^\p.
\end{align*}
Since $n$ is arbitrary, we have $\sup_n\E[W_n^\p]<\infty$, which is a contradiction.
\end{proof}

Next, we prove homogenization in the whole weak disorder phase.

\begin{proof}[Proof of Theorem~\ref{thmx:ew}(i)]
Fix $M>1$ and set $N\coloneqq \lfloor n^{1/3}\rfloor$. We decompose 
\begin{align*}
\X_n^{f}&=Y_n+Z^{\leq M}_n+Z^{>M}_n,\qquad\text{ where}\\
Y_n&\coloneqq n^{-d/2}\sum_{x\in\Z^d} f(x/\sqrt n)(W_n^{0,x}-W_{N}^{0,x})\\
Z_n^{\leq M}&\coloneqq n^{-d/2}\sum_{x\in\Z^d} f(x/\sqrt n) (W_N^{0,x}\1_{W_{N}^{0,x}\leq M}-1)\\
Z_n^{>M}&\coloneqq n^{-d/2}\sum_{x\in\Z^d} f(x/\sqrt n) W_{N}^{0,x}\1_{W_{N}^{0,x}>M}
\end{align*}
Recall that $f$ is bounded and compactly supported. By Theorem~\ref{thmx:phase}(iii),  $(W_n)_{n\in\N}$ is uniformly integrable, hence $\sup_n\|Z_n^{>M}\|_1<\eps$ for $M$ large enough. In addition, we have $W_n\to W_\infty$ in $L^1$ and hence $\|Y_n\|_1<\eps$ for $n$ large enough. Finally,
\begin{align*}
\|Z_n^{\leq M}\|_2^2&=n^{-d}\sum_{x,y\in\Z^d}f(x/\sqrt n)f(y/\sqrt n)\E\left[(W_N^{0,x}\1_{W_N^{0,x}\leq M}-1)(W_N^{0,y}\1_{W_N^{0,y}\leq M}-1)\right]\\
&\leq c\E[W_N\1_{W_N\leq M}-1]^2+c'M^2(2N+1)^dn^{-d/2},
\end{align*}
where we have split the sum depending on whether $|x-y|\leq  2N$ or $|x-y|>2N$ and used that $W_N^{0,x}$ and $W_N^{0,y}$ are independent in the latter case. Hence $\lim_{M\to\infty}\lim_{n\to\infty}\|Z_n^{\leq M}\|_1=0$.
\end{proof}
\end{appendix}

%%%%%%%%%%%%%%%%%%%%%%%%%%%%%%%%%%%%%%%%%%%%%%
%% Support information, if any,             %%
%% should be provided in the                %%
%% Acknowledgements section.                %%
%%%%%%%%%%%%%%%%%%%%%%%%%%%%%%%%%%%%%%%%%%%%%%
\section*{Acknowledgments}
We are grateful to Shuta Nakajima for introducing us to the question. We are also very grateful to Ryoki Fukushima for many interesting discussions about the topic, for carefully reading this manuscript and for many helpful suggestions. We thank Shuta Nakajima, Simon Gabriel, Quentin Berger and Rongfeng Sun for valuable feedback on an earlier version of this manuscript and Rongfeng Sun for pointing out a mistake in that version. Finally, we thank an anonymous referee whose careful reading and helpful comments greatly improved the current article. 

\bibliographystyle{plain}
\bibliography{ref}

\end{document}